\documentclass[11pt, reqno]{amsart}
\usepackage{graphicx,color}
\usepackage{amsmath,amsfonts,amstext,amssymb,amsbsy,amsopn,amsthm,eucal}
\usepackage{txfonts}
\usepackage{tikz-cd}
\usepackage{mathbbol}
\usepackage{calc}
\usepackage{marginnote}
\usepackage{hyperref}
\usepackage{times}
\usepackage{enumerate}

\hypersetup{
  pdftitle={Convex Functional},
  pdfauthor={Zahra Sinaei},
  pdfsubject={},
  pdfkeywords={},
  pdfpagelayout=SinglePage,
  pdfpagemode=UseOutlines,
  colorlinks,
  bookmarksopen,
  linkcolor=[rgb]{0,0,0.7},
  urlcolor=[rgb]{0,0,0.4},
  citecolor=[rgb]{0,0,0.7}
}
\newtheorem{thm}{Theorem}[section]
\newcommand{\bt}{\begin{thm}}
\newcommand{\et}{\end{thm}}

\newtheorem{ex}[thm]{Example}
\newcommand{\bex}{\begin{ex}}
\newcommand{\enx}{\end{ex}}

\newtheorem{cor}[thm]{Corollary}   
\newcommand{\bc}{\begin{cor}}
\newcommand{\ec}{\end{cor}}

\newtheorem{lem}[thm]{Lemma}  
\newcommand{\bl}{\begin{lem}}
\newcommand{\el}{\end{lem}}

\newtheorem{prop}[thm]{Proposition}
\newcommand{\bp}{\begin{prop}}
\newcommand{\ep}{\end{prop}}

\newtheorem{defn}[thm]{Definition}
\newcommand{\bd}{\begin{defn}}      
\newcommand{\ed}{\end{defn}}

\newtheorem{rmrk}[thm]{Remark}   
\newcommand{\br}{\begin{rmrk}}
\newcommand{\er}{\end{rmrk}}

\newcommand{\be}{\begin{equation}}
 \newcommand{\ee}{\end{equation}}

\newtheorem{step}{Step}

\newtheorem{assp}{{Assumption}}

\newcommand{\thmref}[1]{Theorem~\ref{#1}}
\newcommand{\secref}[1]{Section~\ref{#1}}
\newcommand{\subref}[1]{Subsection~\ref{#1}}
\newcommand{\subsubref}[1]{Subsubsection~\ref{#1}}
\newcommand{\appref}[1]{Appendix~\ref{#1}}
\newcommand{\lemref}[1]{Lemma~\ref{#1}}

\newcommand{\propref}[1]{Proposition~\ref{#1}}

\newcommand{\asspref}[1]{Assumption~\ref{#1}}
\newcommand{\eref}[1]{(\ref{#1})}
\newcommand{\Align}[1]{\begin{align}\begin{split} #1 \end{split}\end{align}}
\newcommand{\Aligns}[1]{\begin{align*}\begin{split} #1 \end{split}\end{align*}}
\newcommand{\EQ}[1]{\begin{equation}\begin{split} #1 \end{split}\end{equation}}
\newcommand{\EQL}[1]{\begin{equation}\end{equation}}

\newcommand{\rt}{\right}
\newcommand{\lt}{\left}
\newcommand{\la}{\langle}
\newcommand{\ra}{\rangle}

\newcommand{\pr}[1]{\left(#1\right)}
\newcommand{\bra}[1]{\left[#1\right]}
\newcommand{\cur}[1]{\left\{#1\right\}}
\newcommand{\abs}[1]{\left|#1\right|}
\newcommand{\rhp}{\rightharpoonup }

\newcommand{\mc}{\mathcal}
\newcommand{\mf}{\mathfrak}
\newcommand{\p}{\partial}

\newcommand{\B}[2]{B_{#1}\pr{#2}}

\newcommand{\supp}[1]{\operatorname{supp}\pr{#1}}

\newcommand{\spt}{\operatorname{spt}}

\newcommand{\N}{\mathbb{N}}

\newcommand{\R}{\mathbb{R}}

\newcommand{\na}{\nabla}
\newcommand{\ze}{\zeta}
\newcommand{\om}{\omega}
\newcommand{\al}{\alpha}
\newcommand{\bet}{\beta}
\newcommand{\del}{\delta}
\newcommand{\eps}{\epsilon}

\newcommand{\lam}{\lambda}

\newcommand{\tom}{\tilde{\omega}}
\newcommand{\tdel}{\tilde{\delta}}
\newcommand{\tal}{\tilde{\alpha}}
\newcommand{\tbet}{\tilde{\beta}}
\newcommand{\tDel}{\tilde{\Delta}}

\newcommand{\vol}{\operatorname{vol}}

\newcommand{\diam}{\operatorname{Diam}}
\newcommand{\diver}{\operatorname{div}}

\newcommand{\dvol}{\operatorname{dvol}}

\newcommand{\inj}{\rm{inj}}
\newcommand{\loc}{\rm{loc}}

\newcommand{\sker}{\mc{S}^k_{\epsilon,r}}
\newcommand{\ske}{\mc{S}^k_{\epsilon}}
\newcommand{\skedr}{\mc{S}^k_{\epsilon,\hat {\delta} r_0}}
\newcommand{\sketdr}{\mc{S}^k_{\epsilon,\tilde {\delta} r_0}}
\newcommand{\nab}{\abs{\nabla u}^2}
\newcommand{\czp}{\mc{C}_{r_0^+}}
\newcommand{\cz}{\mc{C}_{r_0}}

\newcommand{\V}{\frac{1}{\vol\pr{\B {r} {x_0}}}}

\begin{document}
\title[Convex functional ]{Convex functional and the stratification of the singular set  of their stationary points}
\author{Zahra Sinaei}
\begin{abstract}
We prove partial regularity of stationary solutions and minimizers $u$ from a set $\Omega\subset \R^n$ to a Riemannian manifold $N$, 
for the functional $\int_\Omega  F(x,u,\nab) dx$. The integrand $F$ is convex and  satisfies some ellipticity and boundedness assumptions. We also  develop a new monotonicity formula and an $\epsilon$-regularity theorem  for such stationary solutions with no restriction on their images.  We then use the idea of quantitative stratification to show that the k-th strata of the singular set of such solutions are k-rectifiable.
\end{abstract}
\maketitle
\section{\label{introduction}Introduction}
In this paper we develop  the regularity theory for minimizing and stationary points of the energy functional 
\EQ{\label{energy-functional}E(u)=\int_\Omega F(\nab) dx}
or more generally
\EQ{\label{energy-functional2}E(u)=\int_\Omega F(x,i\circ u,\nab) dx}
where 
${u}$ is in Sobolev space of maps ${H^1(\Omega,N)}$, 
 $\Omega$ is an open domain with smooth boundary in $\R^n$, and  $N=N^m$ is a compact, smooth manifold with 
 \EQ{\label{bounded-geometry}
 \partial N=\emptyset, ~\inj(N)>\rho>0,~|\sec_N|<k,~\diam_N<D,}
  isometrically embedded in some Euclidean space, $i:N\hookrightarrow \R^q$.   Abusing  notation in many places in this paper we write $u$ instead of $i\circ u$.  For the purpose of regularity, we assume  the $C^2$ function $F$ 
    \EQ{\label{functionG}F(x,z,\mf p)~\text{with} ~x\in\R^n, z\in\R^q, ~\text{and}~{\mf p}\in \R,}
satisfies some ellipticity and integrability assumptions, i.e.
  \begin{assp}\label{asspG} 
\begin{enumerate}[i.]  
\item For some  $\mc{B}>1$, $F$ satisfies the ellipticity condition
\Aligns{ \mc{B}^{-1} \leq F_{{\mf p}{\mf p}}(x,z,{\mf p})\mf p+\frac{nq}{2}F_{{\mf p}}(x,z,{\mf p}) \leq \mc{B}.
}
\item  $\abs{F_{x_l}(x,z,\mf{p})},~\abs{F_{z_k}(x,z,\mf{p})}< \vartheta \mf{p} $, for some positive
 constant $\vartheta$.
 \item $ F_{{\mf p}{\mf p}}(x,z,{\mf p})\geq 0.$
 \item $F$ satisfies the following integrability  conditions 
\Aligns{ &\int_1^\infty \sup_{x,z} F_{{\mf p}{\mf p}}(x,z,{\mf p})\ln{\mf{p}}~d\mf{p}=\mf{C}<\infty.\\
 &\int_0^1 \mf{p}\sup_{x,z} e(x,z,\mf{p}^{-2}) d\mf{p}=\mf{D}<\infty.}
 \end{enumerate}
\end{assp}
Hereafter we always assume $F$ satisfies \asspref{asspG}.  See also  \secref{special-case} for more explanation on condition i. 
The Euler-Lagrange equation for this energy functional is 
\EQ{\label{weakly-harmonic2}
-\int F_{z_k}(x,u,\nab)\ze^k+\int  F_{\mf p}(x,u,\nab)\bra{\la\na_i u,\na_i \ze\ra- A(u)(\na u,\na u)\ze}=0,
}
where  $F_{z_k}$  denotes the partial derivative with respect to the $k$-th component of $z=(z_1,\ldots,z_q)$ and $F_{\mf p}$ denotes the derivative with respect to the last component of $F(x,z,{\mf p})$.  Considering the variation generated by a compactly supported vector field $\ze$  on $\Omega$, the stationary equation related to this energy functional is 
\EQ{\label{stationary-harmonic2}
-\int F_{x_l}(x,u,\nab)\ze^l+\int  F_{\mf p}(x,u,\nab) \la\na_i u,\na_j u\ra\nabla^i\ze^j -F(x,u,\nab)\diver(\ze)=0,
}
where $ F_{x_l}$ denotes the partial derivative with respect to the $l$-th component of $x=(x_1,\ldots,x_n)$.  We call the weak  solutions of \eref{weakly-harmonic2} and \eref{stationary-harmonic2},  {\it{stationary F-harmonic maps}} and the minimizers of the functional  \eref{energy-functional} and \eref{energy-functional2}, {\it{minimizing F-harmonic maps}}.

The existence and regularity of  minimizing and  stationary $F$-harmonic maps have been considered extensively. For example in \cite{Uhlenbeck77}, Uhlenbeck has shown that  under ellipticity assumption on $F$, the weak solutions to 
 equation \eref{weakly-harmonic2}, when $u$ is a map to $\R$,  is $C^{1,\alpha}$ regular for some $0<\alpha<1$.

In \cite{Giaquinta-Modica79}, the authors have shown that  under   smallness assumption on the image, and ellipticity and growth condition on $F$,  
 weak solutions to \eref{weakly-harmonic2} are H\"older continuous outside a set of finite codimension 2  Hausdorff measure.  See also the book \cite{Giaquinta83} and the references therein for a complete survey on this subject. 

Later Schoen and Uhlenbeck in \cite{Schoen-Uhlenbeck82} have developed the classical theory of harmonic maps when
\EQ{\label{classical}F(\nab)=\nab}
and they  have shown that the $k$-dimensional stratum of the singular set
\Aligns{\mc{S}^k(u)=\cur{x\in \Omega~\lt|\rt.~\text{no tangent map at}~x~\text{is k-symmetric}}}
 of the classical stationary harmonic maps,  i.e.  weak solutions of \eref{weakly-harmonic2} and \eref{stationary-harmonic2} for the functional \eref{classical}, satisfy
\EQ{\label{k-stratum}\dim\pr{\mc{S}^k(u)}\leq k.}
They also showed that the singular set of the classical minimizing harmonic maps satisfy 
\EQ{\label{n-3-finite}\mc{S}^{n-3}(u)=\mc{S}(u)}
where $\mc{S}(u)$ denotes the singular set of the map u,
\EQ{\mc{S}(u)=\{x\in\Omega~\lt|\rt.~\exists r>0~\text{such that}~u\lt|\rt._{\B r x}\text{is H\"older continuous}\}^c.}
This was then extended by Lin in \cite{Lin99} where he used the idea of defect measures to prove inequality 
\eref{k-stratum} for  the stationary harmonic maps.  He also showed that 
\EQ{\label{n-2-finite}\mc{H}^{n-2}(\mc{S}(u))=0.}

In the two latter examples the authors prove a monotonicity formula for $\theta(x,r)=r^{2-n}\int_{\B {r} {x}}{\nab}$,
\Aligns{\frac{d}{dr} \theta(x,r)= 2 r^{2-n}\int_{\p \B r x}\abs{\frac{\p u}{\p r}}^2, }
which shows the scale invariant quantity $\theta(x,r)$ is monotone, and is constant if and only if $u$ is homogenous. This is the main step of the proof of inequality \eref{k-stratum}. The proof of \eref{n-3-finite} and \eref{n-2-finite} are again based on  monotonicity formula and a so called $\epsilon$-regularity theorem \cite{Bethuel93}. 

Recently in \cite{Naber-Valtorta17},  Naber and Valtorta have used the idea of quantitative stratification, which  first appeared in the work of   Almgren \cite{Almgren00}  and was later  developed in \cite{Cheeger-Naber13.1}, \cite{Cheeger-Naber13.2} by Cheeger and Naber, to show that when $u$ is stationary harmonic 
\EQ{\label{rectifiablity}\mc{S}^k(u)~\text{ is k-rectifiable}.} 
They  have further shown 
\EQ{\label{finiteness}\mc{H}^{n-3}(\mc{S}^{n-3}(u)\cap \B 1 0)~\text{ is finite}}
when $u$ is a minimizing harmonic map.  

The goal of this paper  is to generalize the  results above for minimizing F-harmonic maps and stationary F-harmonic maps.
A crucial ingredient in the proof of these results was a suitable monotonicity formula. The analogous results could not be extended to stationary solutions and minimizers of the more general functional \eref{energy-functional2} due to the absence of monotonicity formula.
Furthermore, there is no $\eps$-regularity type theorem in this context and for a general target manifold $N$. 

As a crucial  first step  for  proving a regularity result  we obtain   the following monotonicity formula for stationary F-harmonic maps,\EQ{\label{monotonicity2}
\frac{d}{dr}\pr{
e^{\frac{\vartheta}{c_e}r}r^{2-n}\int_{\B {r}{x_0}}F(x,u,\nab) dx+h(r)
}
\geq\int_{\p B_r(x)}F_{\mf p}(x,u,\nab)|\frac{\p u}{\p r}|^2 
}
where $c_e={nq\mc{B}}/{2}$ and $\vartheta$ is a constant depending on $F$. Here $h$ is a positive monotone function with $\lim_{r\to 0} h(r)=0$  which will be defined explicitly in terms of $F$ in \thmref{monotonicity-thmm}.   For the proof of \eref{monotonicity2} we prove a  Jensen-type inequality for functions with positive first derivatives. 

The $\epsilon$-regularity theorem for classical stationary harmonic maps  in \cite{Bethuel93}  says that if  $\theta(x,r)$ is small enough for some positive number $r$, then $u$ is smooth on the ball $\B {r/2} {x}$. By use of similar techniques as in  \cite{Bethuel93}, we  prove the following $\eps$-regularity result. Define 
 \Aligns{\bar\Theta(x_0,r)= {
e^{\frac{\vartheta}{c_e}r}r^{2-n} \int_{\B {r}{x_0}}F(x,u,\nab)~ dx+h(r).
}}
We have
 \bt\label{epsilon-regularityG}
 There exist $\epsilon_0, \alpha \geq 0$ depending only on $n$, $N$ and $F$, such that if  $u\in H^1(B_r(x_0),N)$ is a  stationary F-harmonic map with 
\Aligns{\bar{\Theta}(x_0,r)\leq \epsilon_0,}
then $u$ is in $C^{0,\alpha}(\B {\frac{r}{2}}{x_0})$ with $\abs{u}_{C^{0,\alpha}}\leq C(n,N,F)$.
\et
As  a corollary we show that there exist $\epsilon_0, \alpha, r_0> 0$ depending only on $n$, $N$ and $F$, such that if $u$ is any stationary F-harmonic map  with  
$${r^{2-n} \int_{\B {r}{x_0}}F(x,u,\nab)\leq \epsilon_0,}$$
for some $0<r<r_0$, then $u$ is $C^{0,\alpha}(\B {\frac{r}{2}}{x_0})$. 

Note that in \thmref{epsilon-regularityG}  we only assume $N$ satisfy  \eref{bounded-geometry} and we do not consider any smallness assumption on the image of the map $u$. By a simple covering argument and \thmref{epsilon-regularityG} we get  
$$\mc{H}^{n-2}(\mc{S}(u))=0.$$
Moreover, the monotonicity formula \eref{monotonicity-thmm} and \thmref{epsilon-regularityG} enable us to generalize \eref{rectifiablity} for stationary F-harmonic maps. 
More precisely, for a map $u:\B 3 0\subset\Omega\to N$ with 
\Align{\label{properties-u} u\in H^1(\B 3 0,N),~\text{and}~\bar\Theta_u(0,3)\leq \Lambda.
}
We prove the following result.
\bt\label{mainG} Let $u$ be a stationary F-harmonic map. For every $\eps>0$ there exists, $C_\eps(n,N,\Lambda,F)$ such that for all $0<r\leq 1$ we have 
\EQ{\label{Minkowski1}\vol \pr{   \B {r}{\sker(u)}\cap \B 1 0} \leq C_\eps r^{n-k}.}
Similarly for $\ske$ we have 
\EQ{\label{Minkowski2}\vol \pr{   \B {r}{\ske(u)} \cap \B 1 0 } \leq C_\eps r^{n-k}.}
In particular, $\mc{H}^k\pr {\ske(u)}<C_\eps$.  We also have 
\EQ{\label{Rectifiable1}\ske~ \mbox{is k-rectifiable. }}
As a corollary
\EQ{\label{Rectifiable2}\mc{S}^k~\mbox{is k-rectifiable.}}
\et
Here $\sker(u)$ and  $\ske(u)$ denote the k-th quantitative strata which classify  points on the domain based on 
$L^2$-closeness of the map $u$ to a k-symmetric map in the balls of certain size  around them. 
See \subref{quantitative-singular} for the exact definitions.   For the proof of the above  theorem we follow a similar argument  as in \cite{Naber-Valtorta16}, which uses a simpler covering argument  compared with \cite{Naber-Valtorta17}. Having \thmref{mainG} in hand and by proving a quantitative version of \thmref{epsilon-regularityG}, one can conclude \eref{finiteness} for  minimizing F-harmonic maps and  prove the following theorem.
\bt\label{minimizingG}
 Let $u$ be as in \eref{properties-u} and  be a minimizing F-harmonic map.   Then $\mc{S}(u)$ is $(n-3)$-rectifiable and there exists $C(n,N,\Lambda,F)$ such that
\Aligns{\vol\pr{\B {r}{\mc{S}(u)}\cap \B 1 0}<C r^3.}
Consequently, $\mc{H}^{n-3}(\mc{S}(u)\cap \B 1 0)\leq C$.
\et
We should mention here the results of this paper can be extended to maps from a Riemannian manifold $M$ into a Reimannian manifold $N$, for $N$ as above and where  $M$ satisfies $\inj_M>\rho>0$  and  $|\sec_M |<K_M$.

The organization of this paper is as follows. In  \secref{special-case} we consider the functional \eref{energy-functional} and in \secref{general-case} we generalize the results proven in \secref{special-case}, for  the functional   \eref{energy-functional2}. More precisely
in \subref{monotonicity-special}, we prove a monotonicity formula, \thmref{monotonicity-thm},  which we generalize in \subref{section-monotoneG} for the general functional \eref{energy-functional2}.  \subref{e-regularityS} is where we prove  \thmref{epsilon-regularityG} for \eref{energy-functional}. We adjust this proof for  
\eref{energy-functional2}   in \subref{e-regularityG}. \subref{compactnessS} is devoted to the proof of a compactness result for  solutions of \eref{weakly-harmonic2} and  \eref{stationary-harmonic2},  for the functional \eref{energy-functional} (see \propref{Fatou}) and some properties of tangent maps in this context (see \lemref{tg-map}). We generalize these results in \subref{compactnessG} for the  functional \eref{energy-functional2}.  Finally  we prove   \thmref{mainG} and \thmref{minimizingG} in \secref{stratification}. The proof of \thmref{mainG} requires three additional ingredients. 1. The $L^2$-approximation theorem, \thmref{L2-approximation}, which relates the $\beta$-Jones' number and the average of pinches of the monotone quantity $\bar{\Theta}(x,r)$ on a ball. 2.  Rectifiable-Reifenberg theorems, \thmref{ReifenbergI} and \thmref{ReifenbergII}, which are  generalizations of original Rectifiable-Reifenberg result \cite{Reifenberg60}. 3. Two covering lemmas, \lemref{covering1} and \lemref{covering2}.  Since the proof of these ingredients are similar to the analogous results for  harmonic maps \cite{Naber-Valtorta17} and approximate harmonic maps \cite{Naber-Valtorta16}, we discuss them in  \appref{appendix}.
\section*{Acknowledgments}
I sincerely thank professor Aaron Naber for encouraging me to work on this problem and for many insightful conversations.
\section{\label{special-case}{Special case} $F(\nab)$} 
In this section we consider the energy functional 
\Aligns{
E(u)=\int_\Omega G(\na u)=\int_\Omega F(\nab)
}
on $H^1(\Omega,N)$. The Euler-Lagrange equation with respect to this energy functional is 
\EQ{
\int  F'(\nab)[\la\na_i u,\na_i \ze\ra- A(u)(\na u,\na u)\ze]=0,
}
or equivalently 
\EQ{\label{weak-harmonic}
\diver( F'(\nab) \na u)-F'(\nab) A(u)(\na u,\na u)=0 \quad\text{ in the weak sense}.
}
One can find the stationary points for $E$ considering the variation on $\Omega$ 
\Aligns{\frac{d}{dt}|_{t=0} E(u\circ \phi_t)=
\frac{d}{dt}\lt|\rt._{t=0}E_{(\phi_t)_*g} (u))=0
}
where $\phi_t$ is the flow generated by a compactly supported vector field $X$ and $g$ is the Euclidean metric on $\Omega$. This will reduce to 
\Aligns{\int \bra {F(\nab) g_{\alpha\beta}- 2F'(\nab)(u^*h)_{\alpha\beta} }(\mathrm{L}_X  g^{-1})^{\alpha\beta}.} 
So the stress energy tensor for this equation is 
\Aligns{
S_{\alpha\beta}=F({\nab}) g_{\alpha\beta}- 2F'({\nab})(u^*h)_{\alpha\beta}.
}
$S$ is divergence free
\EQ{
\na^{\alpha }S_{\alpha\beta}=0 \quad\text{in the  distributional sense.}
}
Therefore the stationary equation for the energy functional \eref{energy-functional} is 
\EQ{\label{stationary-harmonic}
\na^{\alpha }\pr{F(\nab) g_{\alpha\beta}- 2F'(\nab)(u^*h)_{~\alpha\beta}}=0 \quad\text{in the distributional sense}
}
or
\EQ{\label{stationary-harmonic1}
\int{F(\nab) \diver(X)- 2F'(\nab)(u^*h)_{~\alpha\beta} \na ^\al X^\bet}=0,
}
for any compactly supported smooth vector field $X$ on $\Omega$. Note that the weak solutions of \eref{weak-harmonic} and \eref{stationary-harmonic} are the stationary $F$-harmonic maps for the functional \eref{energy-functional}.

Without loss of generality we can assume $0 \in\Omega$ and $\B r 0\subset \Omega$ for some $r>0$. Define the vector field $X$ as follows: let $\eta_\eps(\abs{x})$ be a compactly supported smooth function on $\B r 0$ with $\eta_\eps(\abs{x})\equiv 1$  for $x\in \B {r(1-\eps)}{0}$. Then we define $X(x)=\eta_\eps(\abs{x}) x$.  By replacing this vector field in (\ref{stationary-harmonic1}) and sending $\epsilon$ to 0 we have 
\EQ{\label{monotonicity1}
\frac{d}{dr}\pr{ r^{2-n}\int_{\B r 0}F(\nab )}+2r^{1-n}\int_{\B r 0} e(\nab)=2r^{2-n}\int_{\partial \B r 0 }F'(\nab)|\frac{\partial u}{\partial r}|^2,
}
where $e(x)=F'(x)x-F(x)$. We refer to $e$  as the error term.  Note that to obtain the above equation we have not used any assumption on $F$. 
\subsubsection*{\label{F&e1}Properties of $G$ and $F$}  As we mentioned in introduction, we assume some ellipticity  and boundedness assumptions on $F$. Indeed if we assume the integrand $G$ satisfies the following strong ellipticity and boundedness condition,
\EQ{
4\mc{B}^{-1}\abs{\ze}^2 \leq G_{p_{i}^\al p_j^\bet}(p) \ze_i^\al\ze_j^\bet \leq 4\mc{B}\abs{\ze}^2 \quad \text{for all}~ \ze\in M^{n\times q},
}
where $M^{n\times q}$ denotes the space of all real valued matrices,  we  have 
\Aligns{
G_{p_i^\al}(p)&=2F'(\abs{p}^2)p_i^\al,\\
G_{{p_i^\al}{p_j^\bet}}(p)&=4F''(\abs{p}^2)p_i^\al p_j^\bet+2F'(\abs{p}^2)\delta_{ij}\delta_{\al\bet}.
}
By considering $\ze$ the unit vector in $M^{n\times q}$ we have
\EQ{ \label{second-derivative1}
\mc{B}^{-1} \leq F''(x)x+\frac{nq}{2}F'(x) \leq \mc{B}.
}
which is equivalent to the condition i in \asspref{asspG} for the functional  \eref{energy-functional}.
Note that 
\Aligns{ F''(x)x+nq\frac{F'(x)}{2}={x}^{1-\frac{nq}{2}}(x^{\frac{nq}{2}}F'(x))'
}
and therefore 
\Aligns{
\mc{B}^{-1} {x}^{\frac{nq}{2}-1} \leq    (x^{\frac{nq}{2}}F'(x))' \leq \mc{B}  {x}^{\frac{nq}{2}-1}.
}
By integrating the above inequality 
\EQ{ \label{inequaltiy1}\frac{2\mc{B}^{-1}}{nq}\leq F'(x) \leq \frac{2\mc{B}}{nq}} 
and so
\EQ{\label{inequaltiy2}\frac{\mc{B}^{-1}-\mc{B}}{nq}\leq F''(x)x\leq \frac{{\mc{B}-\mc{B}^{-1}}}{nq}.}
Concerning the error term $e$, we have 
and $e$ satisfies the following properies on $[0,\infty)$
\begin{enumerate}[i.]
\item $e(x)$ is bounded for $x<C$.
\item $\displaystyle{\lim_{x\to0}}{e'(x)}=\displaystyle{\lim_{x\to\infty}}{e'(x)}=0$.
 \item $\frac{\mc{B}^{-1}-\mc{B}}{nq}<e'(x)<\frac{\mc{B}-\mc{B}^{-1}}{nq}$.
 \item $e'(x)\geq 0$ if and only if  $F''(x)\geq 0$.
\end{enumerate}
\subsection{\label{monotonicity-special}Monotonicity formula for  the special case}
In this subsection we obtain a monotonicity formula which is  the key for our regularity  theorem. We first recall \asspref{asspG} for the functional  \eref{energy-functional} and through out this section we always assume $F$ satisfies the followings.
\begin{assp}\label{asspS}
\begin{enumerate}[i.]
\item  $F$ satisfies the ellipticity condition \eref{second-derivative1}.
\item  The second derivative of $F$ is non-negative 
\Aligns{F''(x)\geq 0\quad  \text{on}~[0,\infty).}
\item  $F$ satisfies the following integrability condition  
\Aligns{\int_1^\infty  F''(t)\ln(t) dt=\mf{C}<\infty.}
\end{enumerate}
\end{assp}
Before we state our main monotonicity theorem we prove the following lemma which is  crucial ingredients in the proof of this theorem. 
\bl\label{Jensen}
Let  $g(x)$ be a positive function with $g'(x)\geq 0 $ for all $x\geq 0$. Then
\begin{enumerate}[a.]
\item $g$ satisfies the following Jensen-type inequality
\Aligns{\fint_{\B {r} {x_0}} g( f(y)) dy \leq J_g\pr{\fint_{\B {r} {x_0}} f (y)dy}}
for any non-negative $f\in L^1(\Omega)$ and $\B {r} {x_0}\subset \Omega\subset \R^n$, where the function $J_g$ is 
\Aligns{J_g(x)=g(x)+x \int_{x}^{\infty}  \frac{g'(t)}{t}dt.
}
\item If  $\int_1^\infty \frac{g'(t)}{t}\ln(t)<\infty$, then 
\Aligns{\int _0^1 x^{-1} \int_{x^{-2}}^{\infty}  \frac{g'(t)}{t}dt dx<\infty.}
\end{enumerate}
\el
\begin{proof}
To prove part a, for every $y\in \B {r}{x_0}$ we have 
 \EQ{\label{fundamental}g(f(y))-g(\bar{f})=\int_{\bar{f}}^{f(y)} g'(t) dt}
where $\bar{f}=\fint_{\B {r} {x_0}} f(x)dx=\V \int f(x)dx $.  Define the set 
\Aligns{&U=\cur{y \in \B {r} {x_0}~\lt |\rt.~  f(y)\geq \fint f}\\
& U_t=\cur{y \in \B {r} {x_0}~\lt |\rt.~  f(y)\geq t}\\
&U^c=\B {r} {x_0}\backslash U.
}
By averaging  \eref{fundamental} over the ball $\B {r} {x_0}$, 
\Aligns{\fint_{\B{r}{x_0}} g(f) -g(\bar{f})&= \fint _{\B {r} {x_0}}\int_{\bar{f}}^{f(y)} g'(t) dt dy\\
&=\V \int _{U}  \int_{\bar{f}}^{f(y)} g'(t) dt dy-\V \int _{U^c}  \int^{\bar{f}}_{f(y)} g'(t) dt dy\\
&\leq  \V \int _{U}  \int_{\bar{f}}^{f(y)} g'(t) dt dy\\
&= \V \int_{\bar{f}}^{\infty} g'(t) \int_{U_t}  dy ~dt\\
&\leq \V \int_{\bar{f}}^{\infty} g'(t) \int_{U_t}  \frac{f(y)}{t} dy ~dt\\
&\leq  \V \int_{\bar{f}}^{\infty}  \frac{g'(t)}{t} \int_{\B r x}  f(y)dy ~dt= \bar{f} \int_{\bar{f}}^{\infty}  \frac{g'(t)}{t}dt.\\
}
Therefore 
\Aligns{\fint g(f) \leq J(\bar{f}) .
}
For part b we have
\Aligns{\int_0^1 x^{-1} \int_{x^{-2}}^{\infty}  \frac{g'(t)}{t}dt=\int_1^\infty \frac{g'(t)}{t} \int_{t^{-\frac{1}{2}}}^1 \frac{dx}{x} dt\\
= \frac{1}{2}\int_1^\infty \frac{g'(t)}{t}\ln(t)<\infty.
}
\end{proof}
\br
By \asspref{asspS} and  \lemref{Jensen} we have 
\Aligns{\fint_{\B {r} {x_0}} e( f(y)) dy \leq J_e\pr{\fint_{\B {r} {x_0}} f (y)dy}.}
Further, since $\int_0^1 xe(x^{-2})dx=-2x^2F(x^{-2})\lt|\rt._0^1<\infty$ we have 
\Aligns{\int_0^1 xJ_e(x^{-2}) dx<\infty.}
\er
Now we are able to state and prove our monotonicity formula.
\bt\label{monotonicity-thm}
Let $u$ be a stationary F-harmonic map in $H^1(\Omega, N)$ for the functional \eref{energy-functional}. Then there exists $A=A(n,N,\mc{B},\mf{C})$ such that
\EQ{\label{monotonicityf}
\frac{d}{dr}\pr{
r^2 \mathcal{E}(r)+h(r)
}
\geq\int_{\p B_r(x)}{F}'(\nab)|\frac{\p u}{\p r}|^2 
}
where $\mc{E}(r)=\fint_{\B {r}{x_0}}F(\nab) dx$ and $h(r)=2\int_{ 0}^r t J_e(2c_eA^2t^{-2}) dt$ with $c_e={nq\mc{B}}/{2}$. 
\et
\begin{proof}
Recall that we have  
\EQ{\label{monotonicity11}
\frac{d}{dr}\pr{ r^{2}\fint_{\B r x}F(\nab )}+2r J\pr{c_e \fint_{\B r x}F({\nab})}&\geq 
\frac{d}{dr}\pr{ r^{2}\fint_{\B r x}F(\nab )}+2r \fint_{\B r x} e(\nab)\\&=2r^{2-n}\int_{\partial \B r x }F'(\nab)|\frac{\partial u}{\partial r}|^2.
}
where $c_e={nq\mc{B}}/{2}$.
First we claim that $r^2\mc{E}(r)$ is bounded. To prove our claim we will use the following argument.
Let  
\begin{align*}&\hspace{3cm}\mc{E}(1)\leq A^2~~\text{and}\\
&r_0=r_0(A)~\text{ be the smallest}~ r~\text{s.t.} ~r^2\mc{E}(r)\leq 2A^2~\text{on}~[r_0,1].
\end{align*}
We show that there exists $A$ such that $r_0=0$.  By  \eref{monotonicity11}  and for $\bar{r}\in[r_0,1]$ we have  
\EQ{\label{bounded}
\int_{\bar{r}}^1\bra{\frac{d}{dr}\pr{ r^2\mc{E}(r)}+2r J_e(2c_eA^2r^{-2})}\geq 0.}
Put $\alpha=\sqrt{2c_eA^2}$. We have 
\Aligns{ \int_{\bar{r}}^1  rJ_e(\al^2r^{-2}) dr=\al^2 \int _{\bar{r}/\al}^{1/\al} s J_e(s^{-2})ds.
}
Choose $\epsilon\ll 1$. Since 
\Aligns{ \int _0^1 s J_e(s^{-2})ds<\infty,
}
 for $A$ large enough 
\Aligns{ \int _{\bar{r}/\al}^{1/\al} s J_e(s^{-2})ds<\eps/2
}
and therefore
\Aligns{
\int_{\bar r}^1 2r J_e(2c_eA^2r^{-2})<\eps \al^2\ll A^2.
}
Finally by \eref{bounded} we have 
\Aligns{\bar{r}^2\mc{E}(\bar{r})\leq \mc{E}(1)+ \int_{\bar{r}} ^12r J_e(2c_eA^2r^{-2})<2A^2.}
and since $r^2\mc{E}(r)$ is continuous on $(0,1]$,  for $r_0-\del\leq  r\leq r_0$, for some small $\del>0$, we have $r^2\mc{E}(r)\leq 2A^2$ which contradicts  the fact that $r_0$ is the smallest such $r$. Assume $A$ is chosen such that $r_0(A)=0$, then for such $A$, we have
\Aligns{
\frac{d}{dr}\pr{ r^2\mc{E}(r)+h(r)}
={\frac{d}{dr}\pr{ r^2\mc{E}(r)}+2r J_e(2c_eA^2r^{-2})}\geq  \int_{\p B_r(x)}{F}'(\nab)|\frac{\p u}{\p r}|^2. 
}
Note that $\lim_{r\to 0} h(r)=0$ since $\int_0^1 r J_e(2c_eA^2r^{-2})<\infty$.
\end{proof}
\br
An example of a functional which satisfies \asspref{asspS} is 
\Align{ F_1(x)=x\pr{2- \frac{1}{\pr{x+1}^\bet}}}
for $\bet<1$.  
\er
\subsection{\label{e-regularityS}$\epsilon$-regularity theorem for the special case} In this subsection we prove the  $\epsilon$-regularity theorem, \thmref{epsilon-regularityG} for the functional \eref{energy-functional}.  We restate this theorem for this case. Set
\EQ{
\theta(x_0,r)&=r^{2-n}\int_{B_r(x_0)}F(\nab) dx, \\
\Theta(x_0,r)&=\theta(x_0,r)+ h(r).
}
\bt\label{epsilon-regularityS}
There exist $\epsilon_0, \alpha \geq 0$ depending only on $n$, $N$ and $F$ such that if $u\in H^1(B_r(x_0),N)$ is a stationary F-harmonic map for the functional \eref{energy-functional} with 
\Aligns{\label{theta} \Theta(x_0,r)\leq \epsilon_0}
then $u$ is in $C^{0,\alpha}(B_{\frac{r}{2}}(x_0))$ with $\abs{u}_{C^{0,\alpha}}\leq C(n,N,F)$.
\et
Here $\abs{u}_{C^{0,\alpha}(\B {{r}/{2}}{x_0})}=\sup_{x,y\in \B {r/2} {x_0}}\frac{\abs{u(x)-u(y)}}{\abs{x-y}^\al}$.
Before we prove \thmref{epsilon-regularityS},  we  recall some background material which we need for the  proof.
\subsubsection{Background}
Let $\Omega\subset \R^n$ be an open domain with smooth boundary. Let $\phi$ be a positive $L^2$ function on $\Omega$. We will  briefly review Hodge theory on the space $\pr{\Omega, g, \phi dx}$ where $g$  denotes the Euclidean metric, and the Hardy and BMO spaces with respect to the measure $\phi dx$.
\subsubsection*{Hodge theory  on $\pr{\Omega, g, \phi dx}$ }
Let $X$ be a  smooth  vector field on $\Omega$. Then
\Aligns{\int_\Omega\diver(\phi X)dx=\int_\Omega \diver (X)\phi dx+\int_\Omega \la\na\ln\phi,X\ra\phi dx}
We define 
 \Aligns{\diver_ \phi(X)= \diver (X)+\la\na\ln\phi,X\ra=\frac{1}{\phi}\diver(\phi X ).}
 In a similar way,  we define the adjoint operator $\tdel=\del_\phi$ of differential operator $d$ with respect to the measure $\phi dx$ by
 \Aligns{&\int_\Omega \la d\al, \bet\ra\phi dx= \int_\Omega \la \al,\del(\phi \bet)\ra dx\\
  &= \int_\Omega \la \al, \del \bet\ra \phi dx+\int_\Omega \la \al, i_{\na\ln \phi} \bet\ra \phi dx\\
 &=\int_\Omega \la \al,\tdel \beta \ra\phi dx
 }
 where $\tdel\bet= \del\bet+i_{\na\ln \phi} \bet$. 
 \subsubsection*{Hardy and BMO spaces} There is a vast literature on  analysis on spaces of homogeneous type, including Euclidean spaces with doubling measure. These spaces arise  in harmonic analysis in the study of  Hardy-Littlewood maximal functions and Hardy spaces, and duality of Hardy  and BMO spaces (see \cite{Coifman-Weiss71,Coifman-Weiss77}). Many properties of  the classical BMO spase have been shown to hold for doubling metric spaces. These include the  Calder\'on-Zygmund decomposition, the John-Nirenberg inequality (see \cite{Buckley99}).
     
The Hardy  and BMO spaces  on $\R^n$ with respect to the doubling measure $\phi dx$  are defined in a similar way to their original definition, but instead of the Lebesgue measure on $\R^n$, the measure $\phi dx$ is used. We use the notation $\mc{H}^1_\phi$ and $\mathrm{BMO}_\phi$ to distinguish them with their original counterpart.   In this context we also have the following theorems where $\phi dx$ is  a doubling measure.
 \bt\label{Duality}
 Suppose $u \in L^\infty$ and $v\in \mc{H}^1_\phi(\R^n)\cap L^1_\phi(\R^n)$. Then 
 \EQ{
 \int_{\R^n}uv \phi dx\leq C[u]_{\mathrm{BMO}_\phi(\R^n)}\abs{v}_{\mc{H}^1_\phi(\R^n)}.
 }
 \et
\bt\label{Divergence-Free}
Let $f$ be in $H^1_\phi(\R^n)$ and $\om$ be a 1-form  in $L^2_\phi(\R^n)$. Let $\tdel \om=0$ in the distributional sense. Then the function $v=df\cdot\om$ is in the Hardy space $\mathcal{H}^1_\phi(\R^n)$. Moreover, there exists a constant  $C_0$ depending only on $n$ such that  
\EQ{\abs{v}_{\mc{H}^1_\phi(\R^n)}\leq C_0\abs{\om}_{L^2_\phi(\R^n)}\abs{df}_{L^2_\phi(\R^n)}.}
\et
By $L^p_\phi$ and $H^1_\phi$ we mean $L^p$ and $H^1$ spaces with respect to measure $\phi dx$. The proof of the above two theorems follows from the proof of their original counterparts with the Lebesgue measure.
\subsubsection{\label{epsilon-regularitySP}Proof of the \thmref{epsilon-regularityS}}
In this part we prove \thmref{epsilon-regularityS}. The proof is very similar to the proof of original $\epsilon$-regularity theorem for stationary harmonic maps as in \cite{ Bethuel93} (see also \cite{Moser05}). Without loss of generality we assume  $\B 1 0\subset  \Omega$. We denote the inner product on $\B 1 0$  and the space of 1-forms on $\R^n$ by $\cdot$ and the inner product on $N$ by $\la ~,~\ra$.  

By a  similar argument to the one used  in \cite{Helein90, Bethuel93} we can show there exists an orthonormal tangent frame field $\{e_1\circ u,\ldots, e_m\circ u\}$ along the map $u$ which minimizes 
\Aligns{
\frac{1}{2}\sum_{i=1}^m\int_{B_r(x_0)} |\nabla e_i|^2\phi dx
}
where $\phi=F'(\nab)$ and $\B {r}{ x_0}\subset\B 1 0$. Such a minimizer satisfies the following Euler-Lagrange equation in the weak sense 
\EQ{
\diver(\la\na e_i,e_j\ra\phi)=0 \quad \mbox{in }~B_r(x_0),
}
with the Neumann boundary condition 
\Aligns{\la e_i,\frac{\p e_j}{\p \nu}\ra=0.
}
Furthermore, the minimizer satisfies 
\EQ{\label{minimum}
\sum_{i=1}^k\int_{B_r(x_0)}|\na  e_i|^2\phi dx\leq C\int_{B_r(x_0)}\nab\phi dx.
}
Define 
\EQ{\gamma_{ij}=\begin{cases}
\la\na e_i,e_j\ra&\quad x\in B_r(x_0)\\
0&\quad x\notin B_r(x_0).
\end{cases}
}
Then for the 1-form  $\gamma_{ij}$ we have  
\EQ{\label{gamma-df}\diver_\phi \gamma_{ij}=0.}
Since $u$ satisfies (\ref{weak-harmonic}) then for the  1-form $\om_i= \la d  u,e_i \ra $ we have 
\Aligns{
\diver (\phi\om_i)&=\la \diver(\phi\na u),e_i\ra+\phi \la \na u,\na e_i\ra\\
&=\phi \la \na u,e_j \ra \cdot\gamma_{ij}=\phi \om_j\cdot\gamma_{ij}
}
and therefore 
\EQ{
\tilde{\delta} \om_i=\om_j\cdot\gamma_{ij},\\
d \om_i=\om_j\wedge \gamma_{ij}.
}
The following lemma is the main step in the proof of \thmref{epsilon-regularityS}.
\bl\label{IterationS}
There exists a constant $C$ depending on $n$, $N$ and $F$ such that  the following holds. Suppose $u\in H^1(B_r(x_0),N)$
satisfies equation (\ref{weak-harmonic}) with 
\Aligns{r^{2-n} \int_{B_{r}(x_0)}F'(\nab)\nab\leq \epsilon.}
Then
\Aligns{(\kappa r)^{1-n}\int_{B_{\kappa r}(x_0)} |\na u|F'(\nab)  dx
&\leq C\kappa^{1-n}[u]_{\mathrm{BMO}_\phi(B_r(x_0))} (\epsilon+C\sqrt{\epsilon})\\
&+C \kappa r^{1-n}\int_{B_{r}(x_0)} \abs{\na u}F'(\nab) ~dx
}
for any $\kappa\in (0,1)$.
\el
\begin{proof}Consider a compactly supported cut-off function $\eta\in C_{0}^{\infty}(B_r(x_0))$ satisfying $\eta\equiv 1$ in $B_{r/2}(x_0)$ and 
$0\leq \eta\leq 1$ in $B_r(x_0)$, such that $|\na \eta|\leq \frac{4}{r}$. We apply the
Hodge decomposition theorem to  
$$\tom_i=\la d( \eta(u-\bar{u})),e_i\ra,$$
 where $\bar{u}=\fint_{B_r(x_0)} u$, with respect to the measure 
$\phi dx$. Therefore, there exist  $\al_i$ and $\bet_i$ such that
\Aligns{\tom_i=\al_i+\bet_i
}
 where 
 \Aligns{d\al_i=\tdel \bet_i=0}
 and 
\Aligns{\sum _{i=1}^k\int_{\R^n} |\alpha_i|^2+|\beta_i|^2 ~\phi dx\leq C\int _{B_r(x_0)}  |\nabla u|^2~ \phi dx}
 and on $B_{r/2}(x_0)$
 \Aligns{|\nabla u|\leq C\sum_{i=1}^k |\al_i|+|\bet_i|.}
 Further, on  $B_{r/2}(x_0)$ we have 
 \Aligns{\tdel \al_i=\tdel \tom_i=\om_j\cdot\gamma_{ij},\\
 d\bet_i=d\om_i= \om_j\wedge \gamma_{ij}.
 }
 Note that $\alpha_i=d\tal_i$ and $\bet_i=\tdel \tbet_i$ and therefore 
  \Aligns{\tDel \tal_i=\om_j\cdot\gamma_{ij},\\
 \tDel \tbet_i= \om_j\wedge \gamma_{ij}.
 }
 We also have 
 \Aligns{
 \int_{\R^n}\alpha_i\cdot\beta_i ~\phi dx=0.
 }
 \subsubsection*{Estimate for $\bet_i$ } We have
 \EQ{\label{Beta-Control}
 &\int_{\R^n} |\beta_i|^2 ~\phi dx=\int \beta_i\cdot \tom_i ~\phi dx=\int \bet_i\cdot \la d( \eta(u-\bar{u})),e_i\ra~\phi dx\\
& =\int \la \eta(u-\bar{u}) ,\tdel (\bet_i\otimes e_i)\ra~\phi dx
 =\int  \la \eta(u-\bar{u}) ,\bet_i\cdot de_i \ra~\phi dx\\
&  =\int \bet_i\cdot \gamma_{ij} \la \eta(u-\bar{u}) ,e_j \ra~\phi dx
  \leq  C[u]_{\mathrm{BMO}_\phi(B_1(0))}\abs{\beta_i\cdot\gamma_{ij}}_{\mc{H}^1_\phi(\R^n)}\\
&  \leq C[u]_{\mathrm{BMO}_\phi(B_1(0))}  \abs{\beta_i}_{L^2_\phi(\R^n)}\abs{\na e_i}_{L^2_\phi(\R^n)}.
  }
The two last inequalities on the right hand side of \eref{Beta-Control} will  follow  from\thmref{Divergence-Free}, \thmref{Duality},   and the fact that $  [\eta(u-\bar{u})]_{\mathrm{BMO}_\phi(\R^n)}\leq C_1 [u]_{\mathrm{BMO}_\phi(B_1(0))}$, where $C_1$  depends only on $m$ and $n$.

We extend $e_i$ to $\R^n$ such that 
\EQ{
\abs{\na  e_i}_{L^2_\phi(\R^n)} \leq C\int_{B_r(x_0)}\nab~ \phi dx.
}\label{Min}
 Consider now a new cut-off function $\ze\in C_0^\infty(B_{\frac{r}{2}}(x_0))$, $0\leq\ze\leq 1$ with $\ze\equiv 1$ on $B_{\frac{r}{4}}(x_0)$ such that 
 $|\na\ze|\leq \frac{8}{r}$. Then we have 
 \EQ{\label{bet}
 \int_{\R^n}\ze|\bet_i| ~\phi dx&\leq \left(\int_{\R^n} \ze^2~\phi dx\right)^{1/2}\left( \int_{\R^n}|\bet_i|^2~ \phi dx\right)^{1/2}\\
 &\leq C\left(\frac{r}{2}\right)^{n/2}\left( \int_{\R^n}|\bet_i|^2~ \phi dx\right)^{1/2}\\
 &\leq C\left(\frac{r}{2}\right)^{n/2}[u]_{\mathrm{BMO}_\phi(B_{r/2}(x_0))}\left( \int_{B_r(x_0)}\nab~ \phi dx\right)^{1/2},
 }
 where $C$ depends on the upper bound for $F'$ and $n$. 
 \subsubsection*{Estimate for $\al_i$}
 Recall that $\alpha_i=d\tal_i$. Decompose 
\Aligns{\tal_i=\tal_i^1+\tal_i^2,} where 
 \EQ{\begin{cases}
 \tDel \tal_i^1=0& \text{in}~B_{{r}/{2}}(x_0)\\
 \tal_i^1=\tal_i &\text{on}~ \partial B_{{r}/{2}}(x_0)
\end{cases}, }
 and 
 \EQ{\begin{cases}
 \tDel \tal_i^2=\tDel \tal_i=\om_j\cdot\gamma_{ij}& \text{in}~B_{{r}/{2}}(x_0)\\
 \tal_i^2=0 &\text{on}~ \partial B_{{r}/{2}}(x_0).
\end{cases} }
We  estimate first $d\tal_i^2$ as follows
\EQ{\label{tal-i}\int_{B_{r/2}(x_0)} |\na \tal^2_i|~\phi dx\leq \int_{B_{r/2}(x_0)} \na \tal^2_i\cdot \frac{\na \tal_i^2}{|\na \tal_i^2|}~\phi dx\\
\leq \int_{B_{r/2}(x_0)}   \tal^2_i\diver_\phi( \frac{\na \tal_i^2}{|\na \tal_i^2|})~\phi dx.
}
Let  $\psi_i$  be the solution to 
\EQ{\begin{cases}
\tDel \psi_i=\diver_\phi (\frac{\na \tal_i^2}{|\na \tal_i^2|})& \text{in}~B_{{r}/{2}}(x_0)\\
 \psi_i=0 &\text{on}~ \partial B_{{r}/{2}}(x_0).
\end{cases}}
By the Hodge decomposition theorem for $\frac{\na \tal_i^2}{|\na \tal_i^2|}$ and the fact that $\abs{\frac{\na \tal_i^2}{|\na \tal_i|^2}}=1$, we have 
\EQ {\abs{\na \psi_i}_{L^q_\phi}\leq  \abs{\frac{\na \tal_i^2}{|\na \tal_i|^2}}_{L^q_\phi}\leq Cr^{m/q}.
}
The Sobolev embedding theorem implies 
\EQ{\label{Sobolev}
\abs{\psi_i}_{L^\infty(B_{r/2}(x_0))}\leq Cr.
}
So by (\ref{tal-i}) we have
\EQ{\label{tal2}\int_{B_{r/2}(x_0)} |\na \tal^2_i|~\phi dx &\leq \int_{B_{r/2}(x_0)}   \tal^2_i \tDel \psi_i~\phi dx\\
&=\int_{B_{r/2}(x_0)} \tDel\tal^2_i ~\psi_i~\phi dx 
=\int_{B_{r/2}(x_0)} \om_{j}\cdot\gamma_{ij}  ~\psi_i~\phi dx \\
&=\int _{B_{r/2}(x_0)}  \la du,e_j\ra\cdot\gamma_{ij}\psi ~\phi dx
=\int _{B_{r/2}(x_0)} \la u-\bar{u}, \tdel(\gamma_{ij}\otimes \psi_ie_i)~\phi dx\\
&=\int _{B_{r/2}(x_0)} \la u-\bar{u} ,d(\psi_ie_i)\cdot\gamma_{ij}\ra~\phi dx\\
&\leq C[u]_{\mathrm{BMO}_\phi(B_r(x_0))} \left(\int _{B_{r}(x_0)}  |\na u|^2~\phi dx \right)^{1/2}\left(\int _{B_{r}(x_0)}  |d(\psi_i e_i|^2~\phi dx\right)^{1/2}\\
&\leq C[u]_{\mathrm{BMO}_\phi(B_r(x_0))} \left(\int _{B_{r}(x_0)}  |\na u|^2~\phi dx \right)^{1/2}\left(r^2\int _{B_{r}(x_0)} |\na u|^2~\phi dx+Cr^n\right)^{1/2}.
}
Now we estimate $d\tal_i^1$. We  have $\tDel \tal_i^1=0$. Therefore by the  mean value formula in this setting we have 
\EQ{\label{tal1}\int_{B_{\kappa r}(x_0)}|d\tal_i^1|~\phi dx\leq C\kappa^n \int_{B_{ r}(x_0)}|d\tal_i^1|\phi dx.}
By \eref{bet}, \eref{tal1}, \eref{tal2} and for $\kappa\in(0,\frac{1}{2}]$ we have  
\Aligns{\int_{B_{\kappa r}(x_0)}|\na u| \phi dx&\leq \sum_{i=1}^m\int_{B_{\kappa r}(x_0)} (|d\tal_i^1|+|d\tal_i^2|+|\bet_i|~\phi dx\\
&\leq \sum_{i=1}^m\int_{B_{r/2}(x_0)} ( C \kappa^n|d\tal_i^1|+|d\tal_i^2|+|\bet_i|~\phi dx\\
&\leq  \sum_{i=1}^m\int_{B_{r/2}(x_0)} (C \kappa^n|\na u|+|d\tal_i^2|+ |\bet_i|~\phi dx\\
& \leq C_1[u]_{\mathrm{BMO}_\phi(B_r(x_0))} \left(\int _{B_{r}(x_0)}  |\na u|^2~\phi dx \right)^{1/2}\left(r^2\int _{B_{r}(x_0)} |\na u|^2~\phi dx+C_2r^n\right)^{1/2}\\
 &+C_3 \kappa^n \int_{B_{r}(x_0)} |\na u|~\phi dx.
}
  Therefore if $r^{2-n} \int_{B_{r}(x_0)}\nab \phi dx\leq \epsilon$, then we have 
  
\Aligns{&(\kappa r)^{1-n}\int_{B_{\kappa r}(x_0)} |\na u|~\phi dx
\leq 
C_1\kappa^{1-n}[u]_{\mathrm{BMO}_\phi(B_r(x_0))} \left( r^{2-n}\int _{B_{r}(x_0)}  |\na u|^2~\phi dx \right)^{1/2}\left(r^{2-n}\int _{B_{r}(x_0)} |\na u|^2~\phi dx+C_2\right)^{1/2}\\
&+C_3 r^{1-n}\kappa \int_{B_{r}(x_0)} |\na u|~\phi dx\\
&\leq C_1\kappa^{1-n}[u]_{\mathrm{BMO}_\phi(B_r(x_0))} (\epsilon+C_2\sqrt{\epsilon})+C_3 \kappa r^{1-n}\int_{B_{r}(x_0)} |\na u|~\phi dx.
}
\end{proof}
Define 
\Aligns{M(u,x_0,r)=\sup_{B_s(x_1)\subset B_r(x_0)}\left( s^{1-m}\int_{B_s(x_1)}|\na u| ~\phi dx\right).}
\bl\label{Iteration}
There exist $\epsilon_0$  and  $\kappa\in(0,1)$,  depending on $n$, $N$ and $F$, such that the following holds. Suppose $u\in {H}^1(B_r(x_0),N)$ is a stationary F-harmonic map  with 
\Aligns{
\Theta(x_0,r)\leq \epsilon_0. 
}
Then 
\Aligns{M(u,x_0,\kappa r)\leq \frac{1}{2}M(u,x_0,r).}
\el
\begin{proof}
By the monotonicity formula and for every $s\leq \frac{r}{2}$ and $x_1\in B_{\frac{r}{2}}(x_0)$ we have 

\Aligns{\theta(x_1,s)\leq\Theta(x_1,s)\leq \Theta(x_1,\frac{r}{2})\leq c(n)\Theta(x_0,r)\leq c(n)\epsilon_0.
} 
First by \eref{inequaltiy1} we have  
\EQ{s^{2-n}\int_{\B s {x_1}}\nab~\phi dx <Kc(n)\eps_0}
where $K$ depends on $F$ and $n$. Define $\epsilon_1=Kc(n)\epsilon_0$. By the Poincar\'{e} inequality 
\Aligns{[u]_{\mathrm{BMO}_\phi(B_{s}(x_1))}\leq C_PM(u,x_0,r).}
By \lemref{IterationS},
\Aligns{(\kappa s)^{1-n}\int_{B_{\kappa s}(x_1)} |\na u|F'(\nab)  dx
&\leq \left[C\kappa^{1-n}(\epsilon_1+C\sqrt{\epsilon_1})+C \kappa\right]
M(u,x_0,r).
}
We can choose  $\epsilon_1$ and $\kappa$  such that  $C\kappa^{1-n}(\epsilon_1+C\sqrt{\epsilon_1})+C \kappa\leq1/2$, completing the proof of  \lemref{Iteration}.
\end{proof}
Now we are ready to complete the proof of \thmref{epsilon-regularityS}.
\begin{proof} [Proof of \thmref{epsilon-regularityS}]
By \lemref{Iteration} we have 
 $$M(u,x_0,\kappa r)\leq \frac{1}{2}M(u,x_0,r).$$
Applying this lemma repeatedly, and since $M(u,x_0,r)$ is bounded, we have 
$$M(u,x_1,s)\leq C s^{\alpha}, $$ 
for $x_1\in B_{r/2}(x_0)$ and all $s\in(0,r/2]$, where $\alpha$ and $C$ do not depend on $x_1$ and $s$. In particular
\Aligns {\sup_{x_1 \in B_r/2(x_0)}\sup_{0<s\leq r/2} \left( s^{1-m-\alpha}\int_{B_s(x_1)}|\na u| ~\phi dx\right)<\infty.}
Since $\phi$ is bounded and by the Morrey decay lemma (see for example Lemma 2.1 in \cite{Moser05}), $u\in C^{0,\alpha}(B_{r/2}(x_0))$ and $\abs{u}_{C^{0,\alpha}}\leq C(n,N,F)$.
\end{proof}
Here we have another version of \thmref{epsilon-regularityS}.
\begin{lem}\label{Epsilon-Regularity1}
There exist  $\epsilon_0,r_0, \alpha \geq 0$ depending only on $n$, $N$ and $F$ such that if $u\in H^1(B_r(x_0),N)$ is any stationary F-harmonic map with 
\Aligns{\theta(x_0,r)\leq  \epsilon_0}
for $r\leq r_0$, then $u$ is in $C^{0,\alpha}(B_{\frac{r}{2}}(x_0))$ with $\abs{u}_{C^{0,\alpha}}\leq C(n,N,F)$.
\end{lem}
\subsection{\label{compactnessS}Compactness for the special case}
 In order to define the quantitative strata  we first need to build the notion of tangent maps.
 Recall that for a map  $u\in H^1(\Omega,N)$ we define the regular points and  the singular points of $u$ as follows:
\EQ{
&{\mc Reg}(u)=\cur{ x\in \Omega~|~ \text{u is}~C^{0,\alpha}~ \text{in a neighborhood of}~{x}},\\
&{\mc S}(u)={\mc Sing}(u)=\Omega\backslash {\mc Reg}(u),
}
where $\al$ is the minimum of H\"older constants $\alpha$  in \thmref{epsilon-regularityS} and \lemref{Epsilon-Regularity1}. By a simple covering argument and \thmref{epsilon-regularityS}, one can easily show that 
\Aligns{\mc{H}^{n-2}(\mc{S}(u))=0.}
In this subsection we  first study   the convergence of sequences of maps which satisfy \eref{weak-harmonic} and \eref{stationary-harmonic} under a uniform bound on their energy functional, and then we define the notion of tangent maps.  See for example \cite{Schoen84}, \cite{Lin99}  for the similar results for harmonic maps and stationary harmonic maps. More precisely, we have a sequence of maps $u_i$ which satisfies  
\EQ{\label{uniform-energy}\int_{\B 3 0} \int_{\B 3 0}F_i(\abs{\na u_i}^2)<{\Lambda}}
where $F_i$ satisfies \asspref{asspS}.
We also assume $u_i$ satisfies 
\begin{eqnarray}\label{sequenceS1}
\diver\pr{ F_i'(\abs{\na u_i}^2) \na u_i}- F_i'(\abs{\na u_i}^2)A(u_i)(\na u_i,\na u_i)=0\label{weak-harmonic-i}\label{weak-harmonic-i}\label{sequenceS1},\\
\na^{\al}(F_i(\abs{\na u_i}^2) g_{\alpha\beta}- 2F_i'(\abs{\na u_i}^2)u_i^*h_{\alpha\beta})=0\label{sequenceS2},
\end{eqnarray}
in the weak sense.  
\bp\label{Fatou}
Let $u_i$ and $F_i$  be as above. Then there exists a subsequence of $u_i$ (which we still denote by $u_i$) such that
\begin{enumerate}[a.]
 \item $u_i$ converges weakly in $H^1(B_3(0))$   to some $u$ and $u_i$ converges strongly to $u$ in $L^2(B_3(0))$. 
\item  Define 
\Aligns{\Sigma=\displaystyle{\bigcap_{0<r<r_0}}\cur{x\in B_1(0)~|~\liminf_{i} \Theta_{u_i}(x,r)\geq \epsilon_0},}
where $r_0$ is as in \lemref{Epsilon-Regularity1} and $\epsilon_0$ is the minimum of $\eps_0$ in \thmref{epsilon-regularityS} and \lemref{Epsilon-Regularity1}. Then $\Sigma$ is a closed set and has finite $(n-2)$-packing content. 
The maps $u_i$ converge strongly in 
$H^1_{\loc}(B_1(0)\backslash\Sigma, N)\cap C^{0,\alpha}_{\loc} (B_1(0)\backslash \Sigma,N)$ to $u$.
\item $u$ satisfies equation \eref{weak-harmonic-i} weakly  with $F_\infty$. 
\item The Radon measures $\mu_i=F_i(|\nabla u_i|^2)dx $ on $B_1(0)$ converge weakly as Radon measures to $\mu$, 
\Aligns{\mu_i\rightharpoonup \mu.}
 \end{enumerate}
 \ep
 \br
 By Fatou's Lemma 
\Aligns{\mu=F_\infty(\nab) dx+\nu}
where $\nu$ is a non-negative measure on $B_1(0)$ which is  supported on $\Sigma$, 
\Aligns{\Sigma=\spt \nu\cup {\mc S}(u).}
\er
 Define 
\Aligns{
\theta_{\mu_i}(x_0,r)&=r^{2-n}\mu_i({\B {r}{x_0}}),\\
\Theta_{\mu_i}(x_0,r)&=\theta_{\mu_i}(x_0,r)+ h_i(r),
}
where $h_i$ is as in \thmref{monotonicity-thm}. Then 
\Aligns{
\theta_{\mu_i}(x_0,r)&\to \theta_{\mu}(x_0,r)=\theta_{\mu_\infty}(x_0,r)+\theta_\nu(x_0,r)\\
\Theta_{\mu_i}(x_0,r)&\to \Theta_{\mu}(x_0,r)
}
where $\mu_\infty=F_\infty(\nab) dx$. Therefore we have 
\Aligns{\Sigma&=\bigcap_{0<r<r_0}\cur{~x\in B_1(0)~|~\liminf_{i} \Theta_{\mu_i}(x,r)\geq \epsilon_0}\\
&=\cur{x\in B_1(0)~|~\Theta_{\mu}(x)\geq \epsilon_0}\\
&=\cur{x\in B_1(0)~|~\theta_{\mu}(x)\geq {\epsilon_0}}.
}
Here  $\Theta_\mu(x,r)$ is monotone increasing with respect to $r$ and  
\[\Theta_\mu(x)=\lim_{r\to 0} \Theta_\mu(x,r)= \lim_{r\to 0} \theta_\mu(x,r).\]
Note that we do not know if $\theta_\mu(x,r)$ is monotone increasing with respect to $r$ but its limit exists as $r$ goes to zero.

The proof of the above proposition is similar  to Proposition 2.7 in \cite{Naber-Valtorta16}. The key point in the proof of the above proposition  is the following lemma and we leave the rest of the proof to the reader. 
\begin{lem}\label{smalltheta}
Let $u_i$ and $F_i$ be as above. Let $\Theta^i_{u_i}=\Theta^{F_i}_{u_i}\leq \eps_0$ where $\eps_0$ is the same as in Theorem \ref{epsilon-regularityS}.  If $u_i\rhp u$ in $H^1(B_1(0),N)$ then   $u_i$ converges strongly to $u$ in $H^1(B_1(0),N)$ and 
 $u$ satisfies 
\Aligns{
\diver \pr{F_\infty'(\nab)\na u}- F_\infty'(\nab)A(u)(\na u,\na u)=0
}
on $B_1(0)$, in the distributional sense. 
\end{lem}
\begin{proof}
 First by \thmref{epsilon-regularityS}, we have that $\abs{u_i}_{C^{0,\alpha}(B_ {1}(0))}\leq C$, with a uniform bound independent of $i$. Since $N$ is a compact manifold, we also have that $\abs{u_i}_{L^\infty(B_ 3( 0))}$ is uniformly bounded. Thus $u_i$ converges to 
$u$ in  $C^{0,\alpha/2}(\B 1 0)$.
 For the strong $L^2$ convergence we show that 
\Aligns{\int_{B_1(0)}\abs{\na( u_i- u)}^2F'_i(|\na u_i|^2)\ze\to 0
}
for any $\ze$ in $C^\infty_c(B_1(0))$. We have 
\Aligns{&\int_{B_1(0)}|\na( u_i- u)|^2F_i'(|\na u_i|^2)\ze\\
&=\int \la\na( u_i- u),\na u_i\ra F'_i(|\na u_i|^2)\ze+\int \la\na( u_i- u),\na u\ra F'_i(|\na u_i|^2)\ze.
}
The second integral converges to $0$ because of the uniform $C^1$-norm bound on $F_i$ and  since $u_i$ weakly converges to $u$ in $H^1(B_3(0))$.  Further, we have 
\EQ{\label{sthG1}\int \la\na( u_i- u),\na u_i\ra F_i'(|\na u_i|^2)\ze&\leq \int \la u_i- u,\diver(F_i'(|\na u_i|^2)\na u_i )\ra \ze\\
&+C \int \la u_i- u,\na u_i \cdot \na\ze\ra F'(|\na u_i|).
}
This term also converges to zero by the fact $u_i-u$ converges to zero in $L^\infty$ and since 
\EQ{\label{sthG2}
\abs{\int \diver(F'_i(|\na u_i|^2)\na u_i ) }=\abs{F'_i(|\na u_i|^2)A(u_i)(\na u_i,\na u_i)}\leq C \abs{A}_{L^\infty} \int \abs{\na u_i}^2.
}
  To see that $u$ satisfies \eref{weak-harmonic}, note that for every $\ze$ in $C_c^\infty (B_1(0))$ 
\EQ{\label{sthG3}
&\int F_\infty'(|\na u|^2) \la\na u,\na \ze \ra= \lim_{i\to \infty}\int F_i'(|\na u_i|^2)\la \na u_i,\na \ze\ra\\
 &=\lim_{i\to\infty} \int F_i'(\abs{\na u_i}^2)  A(u_i)(\na u_i,\na u_i) \ze=  \int  F_\infty'(\abs{\na u}^2) A(u)(\na u,\na u)\ze.
}
\end{proof}
\subsubsection{Tangent map}\label{tangent-map}
Let $u\in H^1(\B 3 0,N)$ be a stationary F-harmonic map. Define the map $u_{x,\lambda}(y)=u(x+\lambda y)$ for $x\in \B1 0$ and $\lambda\leq 1$. Then the map $u_{x,\lambda}$ satisfies the following equations: 
\EQ{
&\int  [\diver( F_\lambda'(|\nabla u_{x,\lambda}|^2) \na u_{x,\lambda})- F_\lambda'(|\nabla u_{x,\lambda}|^2)A(u_{x,\lambda})(\na u_{x,\lambda},\na u_{x,\lambda})]\ze=0\\
&\int F_\lambda(|\nabla u_{x,\lambda}|^2) \diver(\ze)- 2F_\lambda'(|\nabla u_{x,\lambda}|^2)(u_{x,\lambda}^*h)_{\alpha\beta}\nabla^\al\ze^\beta=0
}   
where $F_\lambda(x)=\lambda^2F(\frac{x}{\lambda^2})$ and so $F_\lambda'(x)=F'(\frac{x}{\lambda^2})$. Note that the corresponding $G_\lambda$ for $F_\lambda$ satisfies \asspref{asspS}.  Then $\theta^\lambda$ and $\Theta^\lambda$ for the function $F_\lambda$  will be as follows
\Aligns{
\theta^\lambda_{u}(x_0,r)&=r^{2-n}\int_{B_r(x_0)}F_\lambda (\nab) dx, \\
\Theta^\lambda_u(x_0,r)&=\theta^\lambda_u(x_0,r)+ h_\lambda (r).
}
One can easily check that  $h_\lambda(r)=h(\lambda r)$, and thus
\Aligns{
&\theta^\lambda_{u_{x,\lambda}}(0,r)= \theta_{u}(x,\lambda r), \\
&\Theta^\lambda_{u_{x,\lambda}}(0,r)= \Theta_{u}(x,\lambda r).
}
By the monotonicity formula for $\theta^\lambda$ we have  
$
\int_{B_1(0)}F_\lambda (|\nabla u_{x,\lambda}|^2)$ {is uniformly bounded.}   

Therefore  there exist a subsequence of $u_{x,\lambda}$ (denoted again by $u_{x,\lambda}$) which converges weakly in $H^1(B_1(0))$ to a map $u_*$ as $\lambda$ goes to zero. We have   
\Aligns{ &\lim_{\lambda\to 0} F_\lambda(x)=\lim_{t\to \infty}F'(t) x=F'(\infty)x,\\
&\lim_{\lambda\to 0} h_\lambda(r)=0.
}
For a measure $\mu$ we define $\mu_{x,\lambda}(A)=\lambda^{n-2}\mu(x+\lambda A)$. For $\mu=F(\nab) dx$, then we have 
\[\mu_{x,\lambda}=F_\lambda(|\nabla u_{x,\lambda}|^2) dx \rightharpoonup \mu_*=F'(\infty)|\nabla u_*|^2dx+\nu_*\]
and 
\Aligns{\theta_{\mu_{x,\lambda}}(0,r)\to \theta_{\mu_*}(0,r)=\theta_{\mu_{\infty,*}}(0,r)+\theta_{\nu_*}(0,r)
}
where $\mu_{\infty,*}=F'(\infty)|\nabla u_*|^2dx$. Note that $\Theta_{\mu_*}(0,r)=\theta_{\mu_*}(0,r)$ and therefore $\theta_{\mu_*}(0,r)$ is monotone increasing with respect to $r$. We further have 
\EQ{\label{tg-measure}
\theta_{\mu_*}(0,r)=\lim_{r\to 0}\theta_\mu(x,r)= \theta_\mu(x).
}
We call $u_*$ a \textit{tangent map} for $u$ at $x$ and  we have the following result.
\bl\label{tg-map}
The tangent map $u_*$ satisfies the following properties:
\begin{enumerate}[a.]
\item $u_*$ is  homogeneous, i.e. $(u_* )_{0,\lambda}=u_*$.
\item  $u_*$ is a weakly harmonic map.  
\item The measures $\mu_*$ and $\nu_*$ are homogeneous measures.
\end{enumerate}
\el
\begin{proof}
For Part a, the 0-homogeneity of the tangent map is because by monotonicity formula \eref{monotonicityf} we have
\EQ{
\int_{ B_1(0)\backslash B_t(0)}
{2r^{2-n}}
 {F}'_\lambda(|\nabla u_{x,\lambda}|^2)|\frac{\partial u_{x,\lambda}}{\partial r}|^2 
\leq \Theta^\lambda_{u_{x,\lambda}}(0,1)-\Theta^\lambda_{u_{x,\lambda}}(0,t)
}
and therefore
\EQ{
\int_{ B_1(0)\backslash B_t(0)}
{2r^{2-n}}
 {F}'(|\nabla u_{*}|^2)|\frac{\partial u_{*}}{\partial r}|^2 
\leq \theta_{\mu_{*}}(0,1)-\theta_{\mu_{*}}(0,t)=0.
}
This shows  that $\frac{\partial u_{*}}{\partial r}=0$ for almost every $r$. Part b is obtained by \propref{Fatou}. 
Part c  follows by a similar argument to the one in \cite{Lin99}, Lemma 1.7 (ii). 
\end{proof}

\section{\label{general-case}{General case} $F(x,u,\nab)$}
In this section we consider the general case   
\Aligns{
E(u)=\int_\Omega F(x,u,\nab)
}
for  $u\in H^1(\Omega,N)$. Recall that the Euler-Lagrange equation with respect to this energy functional is 
\Aligns{
-\int F_{z_k}(x,u,\nab)\ze^k+\int  F_{\mf p}(x,u,\nab)\bra{\la\na_i u,\na_i \ze\ra- A(u)(\na u,\na u)\ze}=0,
}
and the stationary equation related to this energy functional is 
\Aligns{
-\int F_{x_l}(x,u,\nab)\ze^l+2\int  F_{\mf p}(x,u,\nab) \la\na_i u,\na_j u\ra\nabla^i\ze^j -F(x,u,\nab)\diver(\ze)=0.
}
Again by considering $\ze$ as  in \secref{special-case}, for $q\in \Omega$ and $r>0$ such that $\B r q\subset \Omega$ we have
\EQ{\label{monotonicity222}
\frac{d}{dr}\pr{ r^{2-n}\int_{\B r {x_0}}F(x,u,\nab )}&+2r^{1-n}\int_{\B r  {x_0}} e(x,u,\nab)- r^{1-n}\int_{\B r {x_0}}F_{\mf p}(x,u,\nab )(x_l -q_l)\\
&=2r^{2-n}\int_{\partial \B r {x_0}}F_{\mf p}(x,u,\nab)|\frac{\partial u}{\partial r}|^2,
}
where 
\EQ{e(x,u,\nab)=F_{\mf p}(x,u,\nab)\nab-F(x,u,\nab).}
By a similar argument to that  in the special case $F=F(\nab)$,  the ellipticity condition i in \asspref{asspG} imposes the following conditions on $F$:
\EQ{ \label{inequaltiy1-1}\frac{2\mc{B}^{-1}}{nq}&\leq F_{{\mf p}}(x,z,{\mf p}) \leq \frac{2\mc{B}}{nq},\\
\frac{\mc{B}^{-1}-\mc{B}}{nq}&\leq F_{{\mf p}{\mf p}}(x,z,{\mf p}){\mf p}\leq \frac{{\mc{B}-\mc{B}^{-1}}}{nq}.
}
\subsection{Monotonicity formula for the general case}\label{section-monotoneG}
We use a similar argument to that in \lemref{Jensen} and \thmref{monotonicity-thm} to prove our monotonicity formula for the  general case. 
\bl
The error term   $e$ satisfies the following Jensen-type inequality
\Aligns{\fint_{\B {r} {x_0}} e(x,u,\nab) dx \leq J\pr{\fint_{\B {r} {x_0}}\nab dy}}
for any map $u\in H^1(\Omega, N)$ and $\B {r} {x_0}\subset \Omega$, where the function $J$ is 
\Aligns{J(y)=\tilde{E}(y)+y \int_{y}^{\infty}  \frac{E(t)}{t}dt
}
with $\tilde{E}(\mf{p})=\sup_{x,z}e(x,z,\mf{p})$ and $E(\mf{p})=\sup_{x,z}e_{\mf p}(x,z,\mf p)$.  Furthermore
\Aligns{\int_0^1 yJ(y^{-2}) dy<\infty.
}
\el
\begin{proof}
For every $x \in \B {r}{x_0}$ we have
 \EQ{\label{fundamentalG}e(x,u(x),f(x))-e(x,u(x),\bar{f})=\int_{\bar{f}}^{f(x)} e_t(x,u,t) dt}
where $f=\nab$ and  $\bar{f}=\fint_{\B {r} {x_0}} \nab =\V \int \nab dx $.  Define the set 
\Aligns{&U=\cur{x \in \B {r} {x_0}~\lt |\rt.~  f(x)\geq \fint f},\\
& U_t=\cur{x \in \B {r} {x_0}~\lt |\rt.~  f(x)\geq t}.
}
By averaging  \eref{fundamentalG} over ball $\B {r} {x_0}$ 
\Aligns{
&\fint e(x,u(x),f(x)) dx-\fint e(x,u(x),\bar{f})=\fint \int_{\bar{f}}^{f(x)} e_t(x,u,t) dt dx\\
&\leq  \V \int _{U}  \int_{\bar{f}}^{f(y)} e_t(x,u,t) dt dx\\
&= \V \int_{\bar{f}}^{\infty}  \int_{U_t} e_t(x,u,t)  dy ~dt\\
&\leq \V \int_{\bar{f}}^{\infty}\int_{U_t} e_t(x,u,t)  \frac{f(x)}{t} dx ~dt\\
&\leq   \V \int_{\B {r} {x_0} }f(x)\int_{\bar{f}}^{\infty} \frac{ e_t(x,u,t)}{t} dx ~dt\\
&\leq \bar{f} \int_{\bar{f}}^{\infty}  \frac{E(t) }{t}dt,
}
where $E(t)=\sup_{x,u} e_t(x,u,t)$. Define $\tilde{E}(t)=\sup_{x,u} e(x,u,t)$.
Therefore 
\Aligns{
\fint e(x,u(x),f(x)) dx&\leq \fint e(x,u(x),\bar{f})+\bar{f} \int_{\bar{f}}^{\infty}  \frac{E(t) }{t}dt\\
&\leq   \tilde{E}(\bar{f})+\bar{f} \int_{\bar{f}}^{\infty}  \frac{E(t) }{t}dt.
}
Finally we have
\Aligns{\fint e(x,u,\nab) \leq J(\fint \nab),
}
where $J(y)= \tilde{E}(y) +y \int_{y}^{\infty}  \frac{E(t) }{t}dt$.
 The second part of this theorem follows from by \lemref{Jensen}.\end{proof}

\bt\label{monotonicity-thmm}
Let $u$ be a stationary F-harmonic  map in $H^1(\Omega, N)$. Then there exists $A=A(n,N,\mc{B},\mf{C},\mf{D})$ such that
\EQ{\label{monotonicityg}
\frac{d}{dr}\pr{
e^{\frac{\vartheta}{c_e}r}r^2 \mathcal{E}(r)+h(r)
}
\geq\int_{\p B_r(x_0)}F_{\mf p}(x,u,\nab)|\frac{\p u}{\p r}|^2,
}
where $\mc{E}(r)=\fint_{\B {r}{x_0}}F(x,u,\nab) dx$ and $h(r)=2\int_{ 0}^r t J(2c_eA^2t^{-2}) dt$ with $c_e={nq\mc{B}}/{2}$. 
\et
\begin{proof}
Recall that by \eref{monotonicity222} we have 
\EQ{
\frac{d}{dr}\pr{ r^{2-n}\int_{\B r {x_0}}F(x,u,\nab )}&+2r^{1-n}\int_{\B r {x_0}} e(x,u,\nab)- r^{1-n}\int_{\B r {x_0}}F_{x_l}(x,u,\nab )x_l \\
&=2r^{2-n}\int_{\partial \B r {x_0}}F_{\mf p}(x,u,\nab)|\frac{\partial u}{\partial r}|^2
}
where $e(x,u,\nab)=F_{\mf p}(x,u,\nab)\nab-F(x,u,\nab)$. By Assumption B we have 
\Aligns{\abs{- r^{1-n}\int_{\B r {x_0}}F_{x_l}(x,u,\nab )x_l } \leq \frac{\vartheta}{c_e} r^{2-n} \int_{\B r {x_0}}F(x,u,\nab). }
Therefore  
\EQ{\label{monotonicity22}
\frac{d}{dr}\pr{ e^{\frac{\vartheta}{c_e}r}r^{2}\fint_{\B r {x_0}}F(x,u,{\nab})}+2r J\pr{c_e \fint_{\B r {x_0}}F(x,u,{\nab})}&\geq  0.
}
The rest of the proof follows by the exact same argument as in the proof of \thmref{monotonicity-thm}.
\end{proof}
\subsection{\label{e-regularityG}$\epsilon$-regularity for the general case}
In this subsection we prove \thmref{epsilon-regularityG}. Roughly speaking, we show that for a map $u\in H^1(\Omega,N)$ which satisfies \eref{weakly-harmonic2} and \eref{stationary-harmonic2}  where the energy \eref{energy-functional2}  is small,  $u$ is H\"older continuous.
The argument here will be a slight generalization of argument in \subref{epsilon-regularitySP}. First we generalize \lemref{IterationS}. 
\bl\label{IterationG}
 Suppose $u\in H^1(B_r(x_0),N)$
satisfies equation (\ref{weakly-harmonic2}) with 
\Aligns{r^{2-n} \int_{B_{r}(x_0)}F_{\mf p}(x,u,\nab)\nab\leq \epsilon.}
Then there exists a constant $C=C(n,N,F,\Lambda)$ such that  the following holds:
\Aligns{(\kappa r)^{1-n}\int_{B_{\kappa r}(x_0)} |\na u|F_{\mf{p}}(x,u,\nab)  dx
&\leq C\kappa^{1-n}[u]_{\mathrm{BMO}_\phi(B_r(x_0))} (\epsilon+C\sqrt{\epsilon})\\
&+C \kappa r^{1-n}\int_{B_{r}(x_0)} \abs{\na u}F_{\mf{p}}(x,u,\nab) ~dx+C\epsilon \kappa^{1-n}r
}
for any $\kappa\in (0,1)$. 
\el
\begin{proof}
We follow a similar argument to that in  \subsubref{epsilon-regularitySP} and we only mention the changes we need to consider in this general case. First there exists an orthonormal tangent frame 
$(e_1,\ldots,e_m)$ along $u$ which satisfies \eref{minimum} and the forms $\gamma_{ij}$ satisfy \eref{gamma-df}.
For the 1-form $\omega_i$  we get 
\EQ{
&\tilde{\delta} \om_i=\om_j\cdot\gamma_{ij}+\frac{1}{\phi} \la F_z(x,u,\nab),e_i\ra, \\
&d \om_i=\om_j\wedge \gamma_{ij}.
}
where  $\phi=F_{\mf{p}}(x,u,\nab)$. Following the argument in the proof of \lemref{IterationS} we then have 
 \Aligns{&\tdel \al_i=\tdel \tom_i=\om_j\cdot\gamma_{ij}+\frac{1}{\phi} \la F_z(x,u,\nab),e_i\ra,\\
& d\bet_i=d\om_i= \om_j\wedge \gamma_{ij}.
 }
The estimate for $\beta_i$ will remain similar to that in  \subsubref{epsilon-regularitySP}  but the estimate for $\al_i$ changes.  For estimate on $\al_i$, define $\tal_i^1$ similar to  that in  \subsubref{epsilon-regularitySP} and let $ \tal_i^2$ satisfy the following.
 \EQ{\begin{cases}
 \tDel \tal_i^2=\tDel \tal_i=\om_j\cdot\gamma_{ij}+\frac{1}{\phi} \la F_z(x,u,\nab),e_i\ra,& \text{in}~B_{{r}/{2}}(x_0)\\
 \tal_i^2=0 &\text{on}~ \partial B_{{r}/{2}}(x_0).
\end{cases} }
While the estimate for $ \tal_i^1$ will remain the same as in the proof \lemref{IterationS}, for  $\tal^2_i$ we have 
\EQ{\label{tal2}\int_{B_{r/2}(x_0)} |\na \tal^2_i|~\phi dx \leq &C[u]_{\mathrm{BMO}_\phi(B_r(x_0))} \left(\int _{B_{r}(x_0)}  |\na u|^2~\phi dx \right)^{1/2}\left(r^2\int _{B_{r}(x_0)} |\na u|^2~\phi dx+Cr^n\right)^{1/2} \\
& + \int _{B_{r}(x_0)} \la F_z(x,u,\nab),e_i\ra \psi_i~dx .
}
By \eref{Sobolev} and \asspref{asspG} we have   
\Aligns{(\kappa r)^{1-n}\int_{B_{\kappa r}(x_0)} |\na u|~\phi dx
\leq &C_1\kappa^{1-n}[u]_{\mathrm{BMO}_\phi(B_r(x_0))} (\epsilon+C_2\sqrt{\epsilon})+
C_3 \kappa r^{1-n}\int_{B_{r}(x_0)} |\na u|~\phi dx \\
&+\kappa^{1-n} \frac{\vartheta}{c_e}r^{2-n}\int _{\B {\kappa r}{x_0}} \nab\phi~dx.
}
\end{proof}Now we prove \thmref{epsilon-regularityG}. 
\begin{proof}[Proof of \thmref{epsilon-regularityG}]
Define \Aligns{M(u,x_0,r)=\sup_{B_s(x_1)\subset B_r(x_0)}\left( s^{1-m}\int_{B_s(x_1)}|\na u| ~\phi dx\right).}
By a similar argument to that  in \lemref{Iteration} we can choose $\epsilon_0$ and $\kappa$ small enough such that  
\Aligns{M(u,x_0,\kappa r)\leq \frac{1}{2}M(u,x_0,r)+\frac{1}{2}.}
 Applying the \lemref{IterationG} repeatedly we have 
$$M(u,x_1,s)\leq C s^{\alpha} $$ 
for $x_1\in B_{r/2}(x_0)$ and all $s\in(0,r/2]$ where $\alpha$, and $C$ do not depend on $x_1$ and $s$. 
 In particular
\Aligns {\sup_{x_1 \in B_r/2(x_0)}\sup_{0<s\leq r/2} \left( s^{1-m-\alpha}\int_{B_s(x_1)}|\na u| ~\phi dx\right)<\infty.}
Since $\phi$ is bounded and by the Morrey decay lemma,   $u\in C^{0,\alpha}(B_{r/2}(x_0))$.
\end{proof}
\subsection{\label{compactnessG}Compactness  for the general case}
In this subsection we discuss the proof of  \propref{Fatou} and \lemref{tg-map}.  We do not state these results  again here. We only consider  sequences $u_i$  which satisfy  
\EQ{\label{uniform-energy}\int_{\B 3 0}F^i(x,u_i,\abs{\na u_i}^2)<{\Lambda},}
where $F^i$ satisfies \asspref{asspG}, and 
\begin{eqnarray}
-& {F}_{z_k}^i(x,u_i,\abs{\na u_i}^2)+ \diver \pr{{F}_{\mf p}^i(x,u_i,\abs{\na u_i}^2)\na u_i}
-{F}_{\mf p}^i(x,u_i,\abs{\na u_i}^2) A(u_i)(\na u_i,\na u_i)=0 \label{sequenceG1}\\
&{F}_{x_l}^i(x,u_i,\abs{\na u_i}^2)+\na^\al \pr{{F}^i(x,u_i,\abs{\na u_i}^2)g_{\alpha\beta}-2  {F}_{\mf p}^i(x,u_i,\abs{\na u_i}^2)u_i^*h_{\alpha\beta}}=0\label{sequenceG2}
\end{eqnarray}
in the weak sense.  

The only main change in the proof of \propref{Fatou} for the general case compared to the special case  happens in the proof of \lemref{smalltheta}. Here we present the changes which we should consider in the  proof of \lemref{smalltheta} for the general case.

For the strong $L^2$ convergence we  show that 
\Aligns{\int_{B_1(0)}\abs{\na( u_i- u)}^2F_{\mf p}^i(x,u_i,\abs{\na u_i}^2)\ze\to 0
}
for any $\ze$ in $C^\infty_c(B_1(0))$. The argument is the same as in the proof of 
\lemref{smalltheta}  except in \eref{sthG2} we have 
\Aligns{
\abs{\int \diver(F_{\mf p}^i(x,u_i,\abs{\na u_i}^2)\na u_i )}=F_{\mf p}^i(x,u_i,\abs{\na u_i}^2) \abs{A(u_i)(\na u_i,\na u_i)}+ {F}_{z_k}^i(x,u_i,\abs{\na u_i}^2)\leq C(\abs{A}_{L^\infty}+\vartheta) \int \abs{\na u_i}^2.
}
Considering the properties of tangent maps in \lemref{tg-map}, first notice that the maps $u_{x,\lambda}$ in the general case satisfy
\Aligns{
-& {F}_{z_k}^\lambda(x,u_{x,\lambda},\abs{\na u_{x,\lambda}}^2)+ \diver \pr{{F}_{\mf p}^\lambda(x,u_{x,\lambda},\abs{\na u_{x,\lambda}}^2)\na u_{x,\lambda}}
-{F}_{\mf p}^\lambda(x,u_{x,\lambda},\abs{\na u_{x,\lambda}}^2) A(u_{x,\lambda})(\na u_{x,\lambda},\na u_{x,\lambda})=0\\
&{F}_{x_l}^\lambda(x,u_{x,\lambda},\abs{\na u_{x,\lambda}}^2)+\na^\al \pr{{F}^\lambda(x,u_{x,\lambda},\abs{\na u_{x,\lambda}}^2)g_{\alpha\beta}-2  {F}_{\mf p}^\lambda(x,u_{x,\lambda},\abs{\na u_{x,\lambda}}^2)u_{x,\lambda}^*h_{\alpha\beta}}=0
}
in the weak sense, where
\Aligns{F^\lambda(x,z,\mf{p})=\lambda^2F(\lambda x,z,\frac{\mf{p}}{\lambda^2}).}
We also have $\theta^\lambda_u(x_0,r)=e^{\frac{\vartheta}{c_e}r}r^{2-n}\int_{\B r {x_0}} F^\lambda(x,u,\nab)$. The limit function is given by
\Aligns{\lim_{\lambda\to 0}F_\lam(x,z,\mf p)=\lim_{t\to \infty} F_{\mf p}(0,z,t)\mf{p}=F_{\mf p}(0,z,\infty)\mf{p}.}
Finally with the  argument similar to the one in \lemref{tg-map}, we can conclude $u_*$, $\mu_*$ and $\nu_*$ are homogeneous and that  $u_*$  weakly satisfies 
\EQ{\label{compacness-seq}
- {F}_{\mf{p}z_k}(0,u_*,\infty)\na^k u_*\abs{\na u_*}^2 + \diver \pr{F_{\mf{p}}(0,u_*,\infty)\na u_*}
-F_{\mf{p}}(0,u_*,\infty) A(u)(\na u_*,\na u_*)=0.
}
\section{Stratification of the singular Set}\label{stratification}
In this section we  prove \thmref{mainG} and \thmref{minimizingG}. The proof of \thmref{mainG}  is very similar to the proof of Theorem $3.1$ in  \cite{Naber-Valtorta16}. We explain the necessary background material for the proof  of \thmref{mainG} in \appref{appendix} and we use the notations from  \appref{appendix}. 
We first recall the definitions of quantitative strata and their properties (see \cite{Cheeger-Naber13.2} and \cite{Naber-Valtorta17}).
\subsection{Quantitative singular set}\label{quantitative-singular}
Here we give  the definition of $k$-th singular strata and its quantitative version. 
\bd
Given a map $h\in H^1(\R^n,N)$, we say that 
\begin{enumerate}[a.]
\item $h$ is homogeneous with respect to the point $p$ if $h(p+\lambda v)=h(p+v)$ for all $\lambda>0$ and $v\in \R^n$. 
\item $h$ is $k$-symmetric if it is homogeneous with respect to the origin and it has an invariant $k$-dimensional subspace, i.e., if there exists a linear subsapce $V\subset \R^n$ of dimension $k$ such that 
$h(x+v)=h(x)$ for all $x\in \R^n$ and $v\in V$.
\end{enumerate}
\ed
A map $h$ is $0$-symmetric if and only if it is homogeneous with respect to the origin. 
\begin{defn}
Given a map $u$ in $H^1(\Omega,N)$, we say that $B_r(x)\subset \Omega$ is $(k,\epsilon)$-symmetric  for $u$ if there exists a $k$-symmetric function $h$ such that  
\[
\fint_{B_1(0)} |u_{x,r}(y)-h(y)|^2dy\leq \epsilon.
\]
\end{defn}
Now we define a stratification for the singular set  $\mc{S}(u)$ of a stationary map $u$ in $H^1(\Omega,N)$.
\begin{defn}
The $k$-th stratafor $u$ which we denote by 
$\mc{S}^k(u)$ is 
\EQ{
\mc{S}^k(u)=
\{x\in \Omega~\lt|\rt.~ \textrm{no tangent map at}~x~\text{ is}~ k\text{-symmetric}\}.
}
\end{defn}
Using the definition of $(k,\eps)$-symmetry we can define the quantitative stratification based on how the points look at different scales as follows.
\begin{defn}
Given a map $u\in H^1(\Omega,N )$, and $r,\epsilon>0$ and $k\in\{0,\ldots,n\}$ we define the $k$-th $(\eps,r)$-stratification $\mc{S}^k_{\eps,r}(u)$ by
\EQ{
\mc{S}^k_{\eps,r}(u)
=\cur{ x \in \Omega~\lt|\rt.~\text{for no}~r\leq s<1,~ B_s(x)~\text{is}~(k+1,\epsilon)\text{-symmetric w.r.t. }~u.}
}
\end{defn}
Note that $\mc{S}^k_{\eps,r}(u)$ has the following  property 
\EQ{
k'\leq k, ~\epsilon'\geq \epsilon,~r'\leq r \Rightarrow ~ \mc{S}^{k'}_{\eps',r'}(u)\subset \mc{S}^{k}_{\eps,r}(u).
}
Using this fact we define the $k$th $\eps$-stratification by 
\EQ{
\mc{S}^k_{\eps}(u)&=\bigcap_{r>0}\mc{S}^k_{\eps,r}(u)\\
&=\cur {x \in \Omega~\lt|\rt.~\text{for no}~0< r<1,~ B_r(x)~\text{is}~(k+1,\epsilon)\text{-symmetric w.r.t. }~u}
}
Note that 
\EQ{\mc{S}^k(u)=\bigcup_{\eps>0}\mc{S}^k_{\eps}(u).
}
See Lemma 4.3 in \cite{Naber-Valtorta16} for the proof of the above equalities.
\subsection{Proof of \thmref{mainG}}
 First we prove the Minkowski estimates (\ref{Minkowski1})  and (\ref{Minkowski2}). 
\subsubsection*{\small \bf Proof of Minkowski estimates}
We will give the proof of  \thmref{mainG} only  for the  sets $\mc{S}\subset \mc{S}^k_{\eps,\del r}$. Since $\del$ is a constant  depending on $(n,N,\Lambda,F,\eps)$ this does not effect  the conclusion for $\mc{S}\subset \mc{S}^k_{\eps,r}$ except for the size of the constant $C'_\eps$. 
  Therefore we will show 
 \EQ{\vol \pr{   \B {r}{\mc{S}^k_{\eps,\del r}(u)}\cap \B 1 0} \leq C'_\eps r^{n-k}.}
where $\del=\min\cur{\hat\del,\tilde \del}$. We put  $\mc{S}={\mc{S}^k_{\eps,\del r}(u)}\cap \B 1 0$ and by the  monotonicity formula for $\bar\Theta$ we have 
\EQ{
\forall x\in \B 1 0~\text{and} ~\forall r\in[0,1], ~\bar\Theta(x,r)\leq \Lambda'= c(n)\Lambda+c(n,\mf{C}). 
}
Let $E=\sup_{x\in\mc{S}} \bar\Theta(x,1)\leq \Lambda'$. We refine the covering in \lemref{covering2} through an inductive process to get the following covering on $\mc{S}$
\begin{eqnarray}
&\mc{S}\subset \bigcup_{x\in\mc{C}^i}\B {r_x}{x}~\text{with}~\sum_{x\in\mc{C}^i}r_x^k\leq c(n)C_F(n),\label{induction1}\\
&\forall x\in\mc{C}^i, r_x\leq r~\text{or}~  \forall y\in \mc{S}\cap \B {2r_x}{x},\bar\Theta(y,r_x)\leq E-i\del.\label{induction2}
\end{eqnarray}
\subsubsection*{\it{ First step of induction}} This will follow by \lemref{covering2}.
\subsubsection*{\it{ Inductive step}} Assume now that  we have a covering which satisfies (\ref{induction1}) and (\ref{induction2}) for $i=j$. We leave the balls with property $r_x\leq r$ as they are and we use  a rescaled version of the  \lemref{covering2} to cover  again the balls centered at $\mc{C}_j$  which satisfy the drop condition $\bar\Theta(y,r_x)\leq E-j\del$ for all $y\in \B{2r_x}{x}\cap \mc{S}$.  By \lemref{covering2} we have 
\Aligns{\mc{S}\cap \B{r_x}{x} \subset \displaystyle{\bigcup_{y\in \mc{C} ^j_x} }\B {r_y} {y}~\text{and}~ \sum _{y\in \mc{C}_x^j} r^k_y\leq C_F(n)r_x^k,}
where for $y\in \mc{C}^j_x$ either  $r_y=r$ or 
\Aligns{\label{drop}\forall z\in \B {2r_y}{y}\cap \mc{S},~\bar\Theta(z,r_y/10)\leq E-(j+1)\del.}
For the latter case we again  cover again the ball $\B{r_y}{y}$ by the minimal set of balls of radius $\rho(n)r_y$, $\cur{B_ {\rho(n) r_y} {(z_y^l)}}_{l=1}^{c(n)}$.  We then get
\Aligns{ \mc{C}_{j+1}\subset \bigcup_{x\in \mc{C}^j}\bigcup_{y\in \mc{C}^j_x}\bigcup_{l=1}^{c(n)}{z_y^l}
}
and so we have 
\Aligns{\sum_{x\in \mc{C}^{j+1}} r_x^k\leq  \sum_{x\in \mc{C}^{j+1}} \sum_{y\in \mc{C}^j_x } r_y^k\leq (c(n)C_F(n))^{j+1}.}
\subsubsection*{\it{ Conclusion}} We continue the induction at most $\lfloor{E/\del}\rfloor+1$ steps. Then we will have 
$C'_\eps=(c(n)C_F(n))^{\lfloor{E/\del}\rfloor+1}$. The proof of (\ref{Minkowski2}) then follow by (\ref{Minkowski1}). \\

 We now prove $\mc{S}^k_\eps$ and $\mc{S}^k$ are rectifiable. 
\subsubsection*{\small \bf Proof  of Rectifiability} To prove that $\mc{S}^k_\eps$ is rectifiable we use \thmref{ReifenbergI} and \lemref{L2-approximation}. Fix $\mc{S}\subset \mc{S}^k_\eps$. For  each $\del>0$ there exists a subset $E_\del\subset \mc{S}$ with 
$\mc{H}^k(E_\del)\leq \del \mc{H}^k(\mc{S})$ such that $F_\del=\mc{S}\backslash E_\del$ is $k$-rectifiable. To see this first note that by monotonicity formula,  for every $\delta$  there exist  $\bar{r}$ and measurable subset $E\subset \mc{S}$ with the following property:
\begin{eqnarray} 
&\mc{H}^k(E)\leq \del \mc{H}^k(\mc{S})\label{prop1}\\
&\forall x\in F_\del=\mc{S}\backslash E_\del ,~ \bar\Theta(x,10\bar{r})-\bar\Theta(x,0)\leq \del.\label{prop2}
\end{eqnarray}
See \cite{Naber-Valtorta16} for the proof of this statement.   We cover $F_\del$ by balls $\B{\bar{r}}{x_i}$ and then on $F_\del\cap\B{\bar r}{x_i}$ we apply \lemref{L2-approximation}. This is possible because $F_\del\subset \mc{S}^k_\eps$ and in view of \eref{prop2}. For simplicity we renormalize the ball $\B{\bar r}{x_i}$ to the unit ball $\B 1 0$.  For all $x\in F_\del$ and $s\leq 1$  and $\mu=\mc{H}^k\lt|\rt. _{F_\del}$ we have
\Aligns{\pr{\beta_{2,\mu}^k(x,s)}^2\leq  C_Ls^{-k}\int_{\B {s}{x}} W_s( y) d\mu(y).}
We integrate the above and by the fact that $\mc{H}^k(\mc{S}\cap \B r x)\leq C_\eps r^k$ and for  $p\in \B 1 0$,  $s\leq r\leq 1$ we have 
\Aligns{\int_{\B r p}\pr{\beta_{2,\mu}^k(x,s)}^2 d\mu(x) &\leq { C_Ls^{-k} \int _{\B r p}\int_{\B {s}{x}} W_s( y) d\mu(y)d\mu(x)} \\
&  \leq C_L C_\eps \int _{\B {r+s} {p}} W_s(x)d\mu(x).\\
}
Integrating again we get
\Aligns{\int_{\B r p}\int _0^r\pr{\beta_{2,\mu}^k(x,s)}^2 d\mu(x) \frac{ds}{s} 
&  \leq C_L C_\eps  \int _{\B {2r} {p}} \bra{\bar\Theta(x,8r)-\bar\Theta(x,0)}d\mu(x)\\
&\leq  c(n)C_L C_\eps^2 \del (2r)^k.
}
Choosing  $\del\leq \frac{\del_R}{C_L C_\eps^2c(n)}$ prove the $k$-rectifiability of  $F_\del$. Sending $\del$ to zero we get the rectifiability of $\mc{S}_\eps^k$. Since $\mc{S}^k(u)=\bigcup_{i}\mc{S}^k_{1/i}(u)$ we conclude that
$\mc{S}^k(u)$ is rectifiable. 
\subsection{ Minimizing maps}
In this subsection we study the singularities of minimizing F-harmonic maps and  prove \thmref{minimizingG}.  first we recall the definition of minimizing F-harmonic maps. 
\bd
We say a  map $u\in H^1(\Omega,N)$ is minimizing F-harmonic map, if for any ball $\B {r} {p}\subset \Omega$ and for any  $w\in H^1(\B {r} {p},N)$ with $w\equiv u$ in a neighborhood of $\p \B {r} {p}$,
\Aligns{\int_{\B {r} {p}}F(x,u,\nab)\leq\int_{\B {r} {p}}F(x,w,\abs{\na w}^2).
}
\ed
Note that a minimizing F-harmonic map is a stationary F-harmonic map.  In what follows, we develop the quantitative version of the  $\eps$-regularity theorem, \thmref{epsilon-regularityG}, which combined with \thmref{mainG}  leads to the proof of \thmref{minimizingG}.
\subsubsection{Quantitative $\eps$-regularity}
First we define the regularity scale  $r_u(x)$ of a map $u: \Omega\to N$ at a point $x$, which measures how far $x$ is from the singular set of $u$.  Define $r_{0,u}(x)$ to be the maximum of $r>0$ such that $u$ is $C^{0,\alpha}$ on $\B r x$. 
\bd Define the regularity scale $r_u(x)$ by
\EQ{r_u(x)=\max\cur {0<r \leq r_{0,u}(x)~\lt|\rt.~ r^\alpha\sup_{p,q \in \B r x} \frac{\abs{u(p)-u(q)}}{\abs{p-q}^\alpha} },
}
where $\al$ is as in \thmref{epsilon-regularityG}. 
\ed
Note that $r_u(x)$ is scale invariant.  Now we are able to state the quantitative $\eps$-regularity theorem minimizing F-harmonic maps. Let
  \Aligns{
  H_\Lambda=\{u\in H^1(\B 3 0,N)~\text{such that} \int_{\B 3 0}F(x,u,\nab )<\Lambda \}.
  }
\bt\label{q-regularity}
Let $u\in H_\Lambda$ be a minimizing F-harmonic map. Then there exists $\eps(n,N,\Lambda,F)$ such that if  $\B r  p$ is not $(n-2,\eps)$-symmetric, then
\Aligns{
r_u(p)\geq r/2.
}
\et
The proof is similar to the proof of Theorem 2.4 in \cite{Cheeger-Naber13.2}. For the sake of completeness we  mention the steps of the proof as each step is an interesting result on its own.
\begin{proof}
\begin{step}\label{compactness}
Let $u_i\in H_\Lambda$ be minimizing F-harmonic maps. Then $u_i$ converges strongly in $ H^1(\B 3 0,N)$ to a map $u$ which is again minimizing F-harmonic map. 
\end{step}
Arguing as in the case of classical minimizing harmonic maps (see for example Lemma 1 Section 2.9 in \cite{Simon96})
we can show that $u$ is minimizing. The strong $L^2$ convergence will follow by similar argument as the one in 
Proposition 4.6 in \cite{Schoen-Uhlenbeck82}. 
\begin{step}\label{n-2-n} 
For all $\bar\eps$, there exists $\del(n,N,\Lambda,\bar\eps)$ such that if $\B r p\subset \B 3 0$ is $(n-2,\del)$-symmetric for the map $u\in H_\Lambda$, then  $\B r p $ is $(n,\bar\eps)$-symmetric.  Consequently we have 
\Aligns{\mc{S}^{n-3}_{\bar\eps,r}(u)\subset\mc{S}^{n-1}_{\del,r}(u). }
\end{step}
The proof is similar to the  proof of   Lemma 2.5 in \cite{Cheeger-Naber13.2}.
\begin{step}\label{reverse} Let $u\in H_\Lambda$ be a minimizing F-harmonic map. Then there exists $\eps_0(n,N,\Lambda,F)$, such that if 
there exists $c\in N$  with 
\Aligns{\fint_{\B r p} \abs{u-c}^2<\eps_0,}
then $u$ is in $C^{0,\al}(\B {r/2} {p})$.
\end{step}
This step follows  by a simple contradictory argument.\\

Now we are able to prove  \thmref{q-regularity}. We argue by contradiction. Suppose there exists a sequence of minimizing F-harmonic maps $u_i\in H_\Lambda$ such that $\B r p$ is $(n-2,\frac{1}{i})$-symmetric but $r_{u_i}(p)<r/2$. By Step~\ref{compactness} the sequence $u_i$ converges strongly to a minimizing  F-harmonic map $u$ in $H^1$  and the ball $\B r p$ is $(n-2,0)$-symmetric for $u$, and therefore by Step~\ref{n-2-n}, $\B r p$ is $(n,\eps_0)$-symmetric. Finally by Step~\ref{reverse},  $u\in C^{0,\al}(\B {r/2} {p})$ which is a contradiction.
\end{proof}
Now we are able to prove \thmref{minimizingG}.
\begin{proof}[Proof of \thmref{minimizingG}]
By \thmref{q-regularity} we know 
\Aligns{\mc{S}(u) \cap  \B 1 0\subset\mc{S}^{n-3}_{\eps}(u).}
Then by \thmref{mainG},
\Aligns{\vol\pr{\B r {\mc{S}(u)}\cap  \B 1 0 }\leq \vol\pr{\B r {\mc{S}^{n-3}_{\eps}(u)}}<C_\eps r^3,}
which shows that the (n-3)-Minkowski dimension, and therefore the (n-3)-Hausdorff dimension of $\mc{S}(u)\cap  \B 1 0$ is finite. 
\end{proof}
\appendix
\section{\label{appendix}Background for the proof of the \thmref{mainG}}
In this part we explain the background material for the proof of  \thmref{mainG} from \cite{Naber-Valtorta16} and for many details we refer the reader to \cite{Naber-Valtorta16}. Throughout this Appendix we assume that
\EQ{\label{condition-u}u~\text{is a stationary F-harmonic map}~\text{and satisfies}~\eref{properties-u}.}
Before we state the results in this section, we draw your attention to the  the following remark.
\br 
In Proposition 2.7(4) in \cite{Naber-Valtorta16}, the authors prove a so called unique continuation property. They use this property in the proof of 
technical results in \subref{quant-rigidity} to show
\EQ{\label{push-hom}\bar \Theta_u(x,r)-\bar \Theta_u(x,r/2)=0~\text{ if and only if}~\frac{\partial u}{\partial t}=0~ \text{for a.e.}~t\in(0,r].} 
One can avoid the unique continuation property and use 
\EQ{\tilde \Theta_u(x,r)=\int_0^\infty \bar\Theta_u(x,s) \psi(\frac{s}{r})\frac{ds}{r}}
instead of $\bar\Theta_u(x,r)$ in the proof of \thmref{mainG}.  Here $\psi$ is a test function in $C_0^\infty([0,1])$ which is equal to 1 on $\bra{\varepsilon,1-\varepsilon}$ for small enough $\varepsilon$, .
Note that 
\EQ{\frac{d}{dr} \tilde \Theta_u(x,r)=\int_0^\infty  \frac {d}{ds}\bar\Theta_u(x,s)\psi(\frac{s}{r})\frac{sds}{r}}
which shows that $\tilde \Theta_u(x,r)$ is also monotone in $r$.   Moreover we have 
\EQ{\label{push-hom2}\tilde \Theta_u(x,r)-\tilde \Theta_u(x,r/2)=0~\text{ if and only if}~\frac{\partial u}{\partial t}=0~ \text{for a.e.}~t\in(0,r].} 
In the proof of \thmref{mainG} and without loose of generality we assume $\bar\Theta$ satisfy \eref{push-hom}.
\er  
\subsection{Quantitative symmetry}\label{quant-rigidity}
Here we recall the adapted  version of  the  quantitative rigidity theorem and the cone splitting theorem (see \cite{Cheeger-Naber13.1}) in our context.  
\bp(Quantitative rigidity)\label{q-rigidity}
Let $u$ satisfies \eref{condition-u}. Then for every $\epsilon>0$, there exist $\delta_0(n,N,F,\Lambda,\eps)$ such that if for some $x$ in $B_1(0)$
\Aligns{
\bar\Theta_u(x,r)-\bar\Theta_u(x,\frac{r}{2})<\delta_0,
}
then $B_r(x)$ is $(0,\eps)$-symmetric.
\ep
See the proof of Proposition 4.1 in \cite{Naber-Valtorta16}.
This proposition says that if $\bar\Theta(x,\cdot)$ is sufficiently pinched on two consecutive scales (i.e. $\bar\Theta_u(x,r)-\bar\Theta_u(x,r/2)$ is small enough), then $B_r(x)$  will be $(0,\eps)$-symmetric. We call the point $x$ a pinched point for $\bar\Theta_u$.
The following definitions express the quantitative version of linear independence.
\bd
We say  that $\{x_i\}_{i=0}^k$, $x_i\in B_1(0)$, $\rho$-effectively span the $k$-dimensional affine subspace 
\Aligns{V=\mathrm{Span}\{x_1-x_0,\ldots,x_k-x_0\}}
if for all $i=1,\ldots,k$
\Aligns{x_i\notin B_{2\rho}( \mathrm{Span}\{x_1-x_0,\ldots,x_{i-1}-x_0\}).}
\ed
\bd
Given $K\subset B_1(0)$, we say $K$, $\rho$-effectively spans a $k$-dimensional subspace  if there exist $\{x_0,\ldots,x_k\}\subset K$ that $\rho$-effectively spans a 
$k$-dimensional subspace.
\ed
The following theorem is a generalization of \propref{q-rigidity} where we have $k+1$ distinct pinching points.
\bp(Cone splitting)\label{q-splitting}
Let $u$ satisfies \eref{condition-u}. Then for every $\eps,\rho>0$, there exist $\delta_1(n, N, F,\Lambda,\eps,\rho)$ such that if for some $\{x_i\}_{i=0}^k\subset B_r(x)$  with $x$ in $B_1(0)$ and  $0<r\leq 1$ we have 
\begin{enumerate}[a.]
\item $\{x_i\}_{i=0}^k$ $\rho$-effectively span a $k$-dimensional subspace $V$,
\item $\bar\Theta(x_i,r)-\bar\Theta(x_i,r/2)<\del_1$ for all $i$,
\end{enumerate}
then $B_r(x)$ is $(k,\eps)$-symmetric.
\ep
The proof is similar to that \propref{q-rigidity}.  See also Proposition 4.6 and its proof in \cite{Naber-Valtorta16}
\bp(Quantitative dimension reduction)\label{q-reduction}
Let $u$ satisfies \eref{condition-u}. For $\rho,\eps>0$ there exists $\delta_2(n, N, F,\Lambda,\eps,\rho)$ such that the following holds. Let 
\Aligns{H=\cur{y\in B_2(0)~\lt|\rt. ~\bar\Theta(y,1)-\bar\Theta(y,\rho)<\delta_2}.}
If $H$ is $\rho$-effectively spanned by a $k$-dimensional subspace $V$, then
\Aligns{\mc{S}^k_{\eps,\delta_2}(u)\subset B_{2\rho}(V).}
\ep
See also Proposition 4.7 in \cite{Naber-Valtorta16}.
The following theorem says that $\bar\Theta$ is almost constant on the pinched points
\bl\label{pinched}
Let $u$ satisfies \eref{condition-u}.  Assume $\bar\Theta(y,1)<E$ for all $y$ in $\B 1 0$. Then for $\rho,\eta>0$ there exists $\delta_3(n, N,F, \Lambda,\eta,\rho)$ such that the following holds. If
$$H=\cur{y\in B_1(0)~\lt|\rt. ~\bar\Theta(y,\rho)<E-\delta_3}.$$
is $\rho$-effectively spanned by a $k$-dimensional subspace $V$, then
\Aligns{\forall x\in V\cap\B 2 0 ,~ \bar\Theta(x,\rho)<E-\eta.}
Moreover if $k\geq n-1$, then $E\geq \eta$.
\el
See  Lemma  4.10 and its proof in \cite{Naber-Valtorta16}.
We also need the following technical  lemma  for the proof of  \lemref{covering1}.
\bl\label{technical}
Let $u$ satisfies \eref{condition-u}.  For $\rho,~\eps>0$ there exists $\del_4(n,N,F,\Lambda,\eps,\rho)$ such that if 
$$\bar\Theta(0,1)-\bar\Theta(0,1/2)<\del_4$$
and there exists $y$ in $\B 3 0$ such that 
\begin{enumerate}[a.]
\item $\bar\Theta(y,1)-\bar\Theta(y,1/2)<\del_4$,
\item for some $r\in[\rho,2]$, $\B r y$ is not $(k+1,\eps)$-symmetric.
\end{enumerate}
Then $\B r 0$ is not $(k+1,\eps/2)$-symmetric.
\el
\subsection{{Generalized Reifenberg theorem}}\label{Reifenberg-theorem} In this part  we recall two versions of Reifenberg's theorem from \cite{Naber-Valtorta17}.  First we define the Jones' $\beta_2$ number. 
\begin{defn}
Let $\mu$ be a non-negative Radon measure on $B_3(0)$. For any $r>0$ and $k\in N$, the $k$-dimensional Jone's
$\beta$ number, $\beta^k_{2,\mu}$ is defined to measure how close the  support of $\mu$ is to a $k$-dimensional affine subspace. More precisely 
\EQ{\beta^k_{2,\mu}(x,r)=\lt(r^{-2-k}\min_{L^k\subset \R^n}  \int_{\B r x} {d^2(y,L^k)}~d\mu(y)\rt)^{1/2}.}
Here $L^k$ denotes the set of $k$-dimensional affine subspaces of $\R^n$.
\end{defn}
Now we are ready to state the generalized Reifenberg's theorems.  See Theorem 3.3 in \cite{Naber-Valtorta17} for the proof.
\begin{thm} \label{ReifenbergI}There exist dimensional constants $\del_R(n)$ and $C_R(n)$ such that for every  
$\mc{H}^k$-measurable subset {$S\subset B_1(0)$} which satisfies
\EQ{\int_{\B r p}\int_0^r \lt(\beta^k_{2,\mu}(x,s)\rt)^2 \frac{ds}{s} d\mu(x)\leq \del_R(n) r^k}
for each $p \in \B 1 0$ and $r\leq 1$ and $\mu=\mc{H}^k |_S$  we have 
\EQ{\mc{H}^k(S)<C_R(n)r^k~\text{and}~ S~\text{is k-rectifiable}.}
\end{thm}
We also need the following discrete version of Reifenberg's theorem. Here we assume the set $S$ to be a discrete subset of $B_1(0)$ such that the balls  $\cur{\B{r_x/5}{x}}_{x\in S}$ are pairwise disjoint balls where $\B {r_x}{x}\subset \B 2 0$. 
  Define $$\mu=\omega_k{ \sum_{x\in S} {r^k_x}~\delta_x}.$$
\begin{thm} \label{ReifenbergII} There exist dimensional constant $\del_R(n)$ and $C_R(n)$
 such that  if $\mu$ satisfies 
\EQ{\int_{\B r p}\int_0^r\pr{\beta^k_{2,\mu}(x,s)}^2 \frac{ds}{s} d\mu(x)\leq \del_R(n) r^k}
for all $\B r p \subset \B 2 0$, then we have 
\EQ{\sum_{x\in S} r^k_x<C_R(n).}
\end{thm}
 See Theorem 3.4 in  \cite{Naber-Valtorta17} for the proof.
\subsection{$L^2$-approximation theorem}\label{L2-theorem}
In this subsection we state the $L^2$-approximation theorem which  together with a covering argument, are the main ingredients of the proof of \thmref{mainG}. This theorem controls the Jone's $\beta$ number 
from above by the averages of pinches on the ball $\B r x$.
\begin{thm}\label{L2-approximation}
Let $u$ satisfies  \eref{condition-u}.  Let $\B r x$ be a ball with $x \in \B 1 0$ and $r\in(0,1]$. For every $\eps>0$ there exists a constant $C_L(n,N,\Lambda,F,\eps)$ such that if $\B {8r} x$ is $(0,\del_3)$-symmetric but not $(k+1,\eps)$-symmetric then 
\EQ{\pr{\beta^k_{2,\mu}(x,r)}^2 \leq C_L r^{-k} \int_{\B r x} W_r(y) d\mu(y), }
where $\mu$ is a non-negative finite measure on $\B r x$ and 
\EQ{W_r(x)= 2\int_{A_{r,8r}(x)}s^{2-n}F_{\mf p}(x,u,\nab)\abs{\frac{\p u}{\p s}}^2 \dvol= \bar\Theta_{8r}(x)-\bar\Theta_r(x).} 
\end{thm}
\begin{proof}
The proof is similar to that of  Theorem 6.1 in \cite{Naber-Valtorta17}. 
\end{proof}
Note that $\del_3$ in the above theorem is the same as $\delta_3$ in \lemref{pinched}. 
\subsection{Covering lemmas}\label{covering-lemmas}
In this subsection we discuss the two covering lemmas as in \cite{Naber-Valtorta16}.  In the first covering lemma we refine our covering by keeping to refine the cover of the so called good balls. In the second covering lemma we refine our cover by keeping to  refine the cover of so called bad balls. 
\begin{lem}[Covering lemma I] \label{covering1} Suppose  $u$ satisfies in \eref{condition-u}. Fix $\epsilon>0$ and $\rho\leq \rho(n)<100^{-1}$ and $r_0\in(0,1]$. There exist $\hat\del(n,N,F,\Lambda,\eps,\rho)$ and a dimensional constant $C_V(n)$  such that the following is true. Let 
\Aligns{\mc{S}\subset \skedr~\mbox{ and}~E=\sup_{x\in \B{2}{0}\cap \mc{S}} \bar\Theta(x,1).}
Assume  $E\leq \Lambda$. There exists a covering of $\mc{S}\cap \B1 0$ such that 
\Aligns{\mc{S}\cap B_1(0)\subset \displaystyle{\bigcup_{x\in \mc{C}} }\B {r_x} {x}~\text{with}~ r_x\geq r_0~\text{and} \sum _{x\in \mc{C}} r^k_x\leq C_V(n).}
Moreover for each $x\in \mc{C}$ one of the following is satisfied
\begin{enumerate}[a.]
\item $r_x=r_0$
\item The set $H_x=\cur{~y\in\mc{S}\cap \B {2r_x} x  ~\lt|\rt .~ \bar\Theta(y,\rho r_x/10 )>E-\hat\delta~}$ is contained in $B_{\rho r_x/5}(V_x)$ where $V_x$ is an affine subspace of dimension at most $k-1$.
\end{enumerate}
\end{lem}
\br
Without loss of generality  we will consider $\rho=2^{-a}$ and $r_0=\rho^{\bar{j}}$ for $a, \bar{j}\in \N$.
\er
\begin{proof}[Proof of Covering Lemma I] The proof will follow by an inductive covering argument as follows. We will start with $\hat\del$ as in \propref{q-reduction} and then we will determine $\hat\del$ in the induction process. We assume our inductive covering at step $j$ satisfies the followings.
\EQ{\mc{S}\subset \bigcup_{x\in\mc{C}^j}\B {r^j_x} {x} =\bigcup_{x\in\mc{C}_b^j}\B {r^j_x} {x} \cup \bigcup_{x\in\mc{C}^j_g}\B {r^j_x} {x}}
The balls with centers in  $\mc{C}^j_b$ are called {\it{bad balls}} and the ones with  centers in  $\mc{C}^j_g$ are called {\it{good balls}}.
\begin{enumerate}[i.]
\item If $x\in \mc{C}_b^j$ then $r^j_x\geq \rho^j$ and $H_x=\cur{~y\in  \mc{S}\cap \B {2r^j_x} x  ~\lt|\rt .~ \bar\Theta(y,\rho r^j_x/10 )>E-\hat\delta~}$   is contained in $B_{\rho r^j_x/5}(V_x)$, where $V_x$ is a $k-1$-dimensional affine subspace.
\item  If $x\in \mc{C}_g^j$ then $r^j_x=\rho^j$ and $H_x=\cur{~y\in \mc{S}\cap \B {2r^j_x} x  ~\lt|\rt .~ \bar\Theta(y,\rho r^j_x/10 )>E-\hat\delta~}$   is contained in $B_{\rho r^j_x/5}(V_x)$, where $V_x$ is a $k$-dimensional affine subspace.
\item  For all $x\neq y$, $\B {r^j_x/5} x \cap \B {r^j_y/5} y=\emptyset$.
\item For all  $x\in \mc{C}^j$ we have $\bar\Theta(x,r^j_x)>E-\eta$.
\item For all $x\in \mc{C}^j$ and for all $s\in [r^j_x,1]$, $\B s x$ is not $(k+1,\epsilon/2)$-symmetric.
\end{enumerate}
\subsubsection*{\em First step of the Induction} Consider the ball $B_1(0)$.  Let 
\EQ{\label{pinched-points} H=\cur{y\in B_2(0)\cap\mc{S}~\lt|\rt.~ \bar\Theta(y,\rho/10)>E-\hat\del}.}
If there exists no $k$ dimensional subspace $V$ such that $H$ is contained in $\B {\rho/5}{V}$ then we call $B_1(0)$ a bad ball and we stop the induction process. Otherwise $B_1(0)$ is a good ball and by \propref{q-reduction}, for $\hat\delta\leq \delta_2$  
 \[\mc{S}^k_{\eps,\hat\delta r}(u)\cap \B 1 0\subset \B {\rho/5}{V}.\]
Now we cover $\B {\rho/5}{V}$ by balls $\cur{\B {\rho} x}_{x\in \mc{C}}$ such that 
\begin{enumerate}[i.]
\item $x\in V\cap \B 10$
\item if $x\neq y$, $\B {\rho/5} x\cap \B{\rho/5} y=\emptyset$.
\end{enumerate}
Then by \lemref{pinched} and for every $\eta$ there exists $\del_3(\eta)$ such that if $\hat\delta\leq{\min\cur{\del_2,\del_3(\eta)}}$  and for every $x\in \mc{C}$ we have 
\Aligns{\bar\Theta(x,\rho/10)>E-\eta.}
Next by \lemref{technical} and for $\rho=\hat\del r_0$ and  every $\eps$ there exists $\del_4(\rho,\eps)$ such that if  $\eta<\min\{\del_4,\del_2\}$ we get 
$\B {s}{x}$ for $s>\rho$ is not $(k+1,\eps/2)$-symmetric. Therefore we have propertes i-v for the first step. 
\subsubsection*{\em Inductive step} In this step we assume we have properties i-v for step $j$ and we prove it for step $j+1$. This is very similar to the first step and we refer the reader to \cite{Naber-Valtorta16}.  
\subsubsection*{\em Conclusion} We stop inductive covering  when $j=\bar{j}$ (recall $r_0=\rho^{\bar{j}}$).
\subsubsection*{\bf \small Volume estimate} Now we prove  the volume estimate
\EQ{
\sum _{x\in \mc{C}} r^k_x\leq C_V(n).
}
Define $$\mu=\omega_k\sum_{x\in\mc{C}} r^k_x \del_x.$$
Therefore it is enough to prove 
\EQ{\label{volume}\mu(\B r x)\leq C_V(n) r^k}
for $x\in \B 1 0$ and $r\leq 1$. 
The proof of (\ref{volume}) will be  based on an inductive  process as follows.  For all $t\in(0,1]$, set 
\Aligns{\mc{C}_t&=\{x\in\mc{C}~\lt|\rt. r_x\leq t\},\\
\mu_t&=\mu\lt|\rt._{\mc{C}_t}\leq \mu. 
}
First we have $\mu_1=\mu$.
For $t_l=2^l r_0\leq 1/8$ we show by induction on $l\geq 0$ that 
\EQ{\label{volume-induction}
\mu_{t_l}(\B {t_l} x)\leq C_R(n) t_l^k,
}
for the constant $C_R(n)$ as in \thmref{ReifenbergII}. We then cover  $B_1(0)$ with $c(n)$ balls of radius $1/8$. We put then $C_V(n)=c(n)C_R(n)$.

The first step of induction is clear since 
$$\mc{C}_{r_0}=\{x\in \mc{C}~\lt|\rt.~r_x=r_0\}$$ 
and they are at least $2r_0/5$ away from each other and therefore 
$$\mu_{r_0}(\B {r_0}{x})\leq c(n) r_0^k.$$

Now assume (\ref{volume-induction}) is true for $t\leq t_l$, we want to show (\ref{volume-induction}) for $t_{l+1}=2t_l$. First we show the following weaker estimate for $\mu_{t_{l+1}}$ and then we improve our estimate by use of $L^2$-approximation Theorem and Reifenberg Theorem. We claim
\EQ{\label{bad-estimate}
\mu_{t_{l+1}}(\B {t_{l+1}} x)\leq c(n) C_R(n) t^k_{l+1}.
}
To prove above we set
\Aligns{\mu_{t_{l+1}}=\mu_{{t_l}}+\tilde\mu_{t_{l+1}}
}
where $\tilde\mu_{t_{l+1}}=\mu \lt|\rt. _{\cur{x\in \mc{C}_{t_{l+1}}~ \lt|\rt.~ r_x>t_l }}$. Take a cover $\B {2t_l} x$ by balls $\cur{\B {t_l} {y_i}}$ such that 
$\cur{\B {t_{l}/2} {y_i}}$ are disjoint. There are $c(n)$ of such balls. Then 
$$
\mu_{t_l}\pr{\B {t_{l+1}} x}\leq \sum_i\mu_{t_l}\pr{ \B {t_l}{y_i}}\leq c(n)C_R(n) t_l^k.
$$
For  $\tilde{\mu}_{t_{l+1}}$ we have 
$$\tilde{\mu}_{t_{l+1}}(\B {t_{l+1}} {x})\leq c(n) t_{l+1}^k$$
since our covering balls $\cur{\B {r_x} x }_{x\in\mc{C}}$ are $2r_x/5$-away from each other. Therefore we have (\ref{bad-estimate}). 

Now we use our inductive assumption in (\ref{volume-induction}) and (\ref{bad-estimate}) and  \thmref{ReifenbergII},  \thmref{L2-approximation} to finish the proof of the volume estimate (\ref{volume}). Set $$\bar{\mu}=\mu_{t_{l+1}}\lt|\rt._{\B {t_{l+1}} x}.$$ 
We use \thmref{L2-approximation} to show
\EQ{\label{L-app}\pr{\beta^k_{2,\bar{\mu}}(y,s)}^2\leq C_Ls^{-k}\int_{\B s y} \hat{W}_s(q) d\bar{\mu}(q)}
for $y\in \supp{ \bar{\mu}}$ and  $s\in(0,1]$ where
\Aligns{ \hat{W}_s(q)=
\begin{cases}
0& s\leq r_q\\
W_s(q)& s>r_q.
\end{cases}
}

We show  (\ref{L-app}) first  for the case where $s\in [r_y,1/8]$. For every $y\in\supp{\bar{\mu}}$ and $s\in[r_y,1/8]$ we know that  
$B_{8s}(y)$   is not $(k+1,\eps/2)$-symmetric.  Moreover, $\bar\Theta(y,r_y)>E-\eta$, and therefore 
\Aligns{\bar\Theta(y,8s)-\bar\Theta(y,4s)<\eta.}
By choosing $\eta\leq \del_0$ where  in \propref{q-rigidity} we use $\eps=\del_3$, we conclude that  $B(y,8s)$ is $(0,\del_3)$-symmetric. Finally, for every $q\in \B s y$ and since $s>r_y$ we have $r_q<s$, and  (\ref{L-app}) follows for the case where $s\in [r_y,1/8]$. Inequality \eref{L-app} is straightforward for  $s\leq r_y$, since  we have  $\B s y \cap \supp{ \bar\mu}=\cur{y}$ and so  $\beta_{2,\bar{\mu}}(y,s)=0$.  

In order to finish our induction we use (\ref{L-app}) to show  that for $y\in \B{t_{l+1}}{x}$ and $r<t_{l+1}$,
\EQ{\label{Reif}\int_{\B r y} \int_0^r\pr{\beta^k_{2,\bar\mu}(z,s)}^2 \frac{ds}{s} ~d\bar\mu(z)\leq C_L c(n)C_R^2\eta r^k.
}
Then by \thmref{ReifenbergII} and for $\eta<\frac{\del_R}{C_L c(n){C_W}^2}$ we have
\Aligns{\mu_{t_{l+1}} \pr{\B {t_{l+1}} {x}}    \leq C_R(n) t^k_{l+1}.}
For the proof of (\ref{Reif}), by (\ref{L-app})  and for all $s\leq r$,
\EQ{\int_{\B r y}\pr{\beta^k_{2,\bar{\mu}}(z,s)}^2 d\bar{\mu}(z)\leq C_Ls^{-k}\int_{\B r y}\int_{\B s z} \hat{W}_s(q) d\bar{\mu}(q)d\bar{\mu}(z).}
By (\ref{bad-estimate}) for $s\leq t_{l+1}$, 
$$\bar{\mu}(\B s z)\leq \mu_{s}(\B {s}{z})\leq c(n)C_R s^k.$$
Thus we have 
\EQ{\int_{\B r y}\pr{\beta^k_{2,\bar{\mu}}(z,s)}^2 d\bar{\mu}(z)\leq c(n)C_R C_L\int_{\B {r+s} y}\hat{W}_s(z) d\bar{\mu}(z)}
and  
\EQ{\int_0^r\int_{\B r y}\pr{\beta^k_{2,\bar{\mu}}(z,s)}^2 d\bar{\mu}(z)\frac{ds}{s}&\leq c(n)C_R C_L\int_0^r\int_{\B {r+s} y}\hat{W}_s(z) d\bar{\mu}(z)\frac{ds}{s}\\
&=c(n)C_R C_L\int_{\B {r+s} y}\int_{r_z}^r\hat{W}_s(z)\frac{ds}{s} d\bar{\mu}(z)\\
&\leq c(n)C_R C_L \int_{\B {r+s} y}\int_{r_z}^{1/8}\hat{W}_s(z)\frac{ds}{s} d\bar{\mu}(z)\\
&\leq c(n)C_R C_L c\bra{\bar\Theta(y,1)-\bar\Theta(y,r_y)}\leq c(n)C_R C_L c\eta (2r)^k. 
}
Therefore if we choose $\eta\leq\frac{ \del_R}{c(n)C_R C_L c}$, then by \thmref{ReifenbergII}, we
get \eref{volume-induction}.
\end{proof}
In our second covering lemma we refine the covering of the  bad balls from Covering Lemma I.
\bl[Covering lemma II] \label{covering2}  Let $u$  satisfies  \eref{condition-u}. Fix $\epsilon>0$ and $r_0\in(0,1]$. There exist $\tilde{\del}=\tilde\del(n,N,\Lambda,F,\eps)$ and a dimensional constant $C_F(n)$  such that the following is true. Let 
\[\mc{S}\subset \sketdr~\mbox{ and}~E=\sup_{x\in \B{2}{0}\cap \mc{S}} \bar\Theta(x,1).\]
Assume  $E\leq \Lambda$. There exists a covering of $\mc{S}\cap \B1 0$ such that 
\Aligns{\mc{S}\cap B_1(0)\subset \displaystyle{\bigcup_{x\in \mc{C}} }\B {r_x} {x}~\text{with}~ r_x\geq r_0~\text{and} \sum _{x\in \mc{C}} r^k_x\leq C_F(n).}
Moreover for each $x\in \mc{C}$ one of the following is satisfied:
\begin{enumerate}[a.]
\item $r_x=r_0$,
\item or we have the following drop
\EQ{\label{drop}\forall y\in \B {2r_x}{x}\cap \mc{S},~\bar\Theta(y,r_x/10)\leq E-\tilde\del.}
\end{enumerate}
\el
\begin{proof}[Proof of Covering Lemma II] We will refine the covering of Lemma \ref{covering2} through an inductive process. At the step $j$ of our induction we have
\begin{enumerate}[i.]
\item  For all $j$
\EQ{\mc{S}\subset \bigcup_{x\in \mc{C}^{(j,r_0)}} \B {r_0}{x}\cup \bigcup_{x\in \mc{C}^{(j,f)}}\B{r_x}{x} \cup \bigcup_{x\in \mc{C}^{(j,b)}}\B{r_x}{x}.
}
\item For all ${x\in \mc{C}^{(j,r_0)}}$, $r_x=r_0$. On these balls condition $a$ of \lemref{covering2} is satisfied and we stop the inductive process.
\item For all ${x\in \mc{C}^{(j,f)}}$, and all $z\in \B{2r_x}{x}$ we have $\bar\Theta(z,r_x/10)\leq E-\tilde{\del}$. 
On these balls condition $b$ of \lemref{covering2} is satisfies  and we stop the inductive process.
\item  For all $x\in \mc{C}^{(j,b)}$, $r_0<r_x\leq\rho^j$ and neither condition a nor condition  b satisfies and we  continue our inductive process.
\item For some constant $C_F(n)$ we have 
\EQ{
\sum_{x\in \mc{C}^{(j,r_0)}\cup \mc{C}^{(j,f) }} r_x^k&\leq C_F(n) \sum_{l=1}^j 2^{-l}\\
\sum _{x\in \mc{C}^{(j,b)} }r_x^k&\leq 2^{-j}.
}
\end{enumerate}
\subsubsection*{\it First step of induction}
Consider  $\tdel\leq \hat\del$, where the exact value of $\tdel$ will be determined  during the proof.  Recall that from \lemref{covering1} that we have the following covering of $\mc{S}\subset \skedr$
\EQ{
\mc{S}\subset\bigcup_{x\in\mc{C}}\B{r}{x}=
\bigcup_{x\in\mc{C}_{r_0}}\B{r}{x}\cup\bigcup_{x\in\czp}\B{r}{x},
}
where
\EQ{ 
\cz=\cur{x\in \mc{C}~\lt|\rt. r_x=r_0} ~\text{and}~\czp=\cur{x\in \mc{C}\lt|\rt. r_x>r_0}
}
{and} $$\displaystyle\sum_{x\in\cz\cup\czp} r_x^k<C_V(n),$$ 
and for every $x\in \czp$ we have
\Aligns{
H_x=\cur{y\in\B{2r_x}x \cap \mc{S} ~\lt|\rt.~\bar\Theta(y,\rho r_x/10)>E-\hat{\delta}}\subset \B {\rho/5}{V_x}
}
for a subspace $V_x$ of dimension at most $k-1$. We include the balls $\{\B {r_x}{x}\}_{x\in \cz}$ in our final covering. In fact  we let
\Aligns{
 \mc{C}^{(1,r_0)}= \mc{C}_{r_0}.
}
For the balls $\cur{\B {r_x} {x}}_{x\in \czp}$ we use a finer cover as follows.

If $H_x=\emptyset$ then every point in $\B {r_x}{x}$ satisfies the drop condition (\ref{drop}) in \lemref{covering2}.   We cover $\B{r_x}{x}$ by balls of radius $\cur{\B{\rho r_x} {y}}_{\mc{C}_x^{1,f}}$. The number of these balls is bounded by a constant $c(n)\rho^{-n}$. 

If $H_x\neq\emptyset$ then
 \Aligns{
H_x\subset \B {\rho/5}{V_x}
}
where $V_x$ is a subspace with dimension  at most  k-1. Then $\B {r_x}{x}\backslash\B {\rho/5}{V_x}$ can be covered by balls of radius $\{\B {\rho r_x} {y}\}_{\mc{C}_x^{(1,f)}}$ as above. On balls $\{\B {\rho r_x} {y}\}_{\mc{C}_x^{(1,f)}}$ the energy drop condition (\ref{drop}) satisfied.  We cover 
$\B {\rho/5}{V_x}$ by  balls  $\{\B {\rho r_x} {y} \}$ and there are at most $c(n)\rho^{1-k}$ such balls. These balls either satisfy the stopping condition $\rho r_x=r_0$,  in which case we include them  in  $\{\B {\rho r_x} {y} \}_{\mc{C}_x^{(1,r_0)}}$,   or they satisfy $\rho r_x>r_0$ where  they need more refinement and we include them in $\{\B {\rho r_x} {y} \}_{\mc{C}_x^{(1,b)}}$ .  More precisely, \Aligns{
\mc{C}^{(1,f)}\subset\displaystyle{\bigcup_{x\in \czp  }} \mc{C}_x^{(1,f)}~\text{and}~
\sum_{z\in \mc{C}^{(1,f)}} r_z^k=\sum_{ x\in \czp}\sum_{y\in \mc{C}_x^{1,f}}(\rho r_x)^k\leq C_V(n)c(n)\rho^{k-n} ,}
\Aligns{
\mc{C}^{(1,b)}\subset\displaystyle{\bigcup_{x\in \czp,~\rho r_x>r_0  }} \mc{C}_x^{(1,b)}~\text{and}~
\sum_{z\in \mc{C}^{(1,b)}} r_z^k=\sum_{ x\in \czp}\sum_{y\in \mc{C}_x^{1,b}}(\rho r_x)^k\leq C_V(n)c(n)\rho, }
\Aligns{
\mc{C}^{(1,r_0)}=\mc{C}_{r_0}\cup\displaystyle{\bigcup_{x\in \czp,~\rho r_x=r_0  }} \mc{C}_x^{(1,r_0)}~\text{and}~
\sum_{z\in \mc{C}^{(1,r_0)}} r_z^k\leq  \sum_{x\in \mc{C}_{r_0}} r_x^k + \sum_{ x\in \czp}\sum_{y\in \mc{C}_x^{1,r_0}}r_0^k\leq C_V(n)+ C_V(n)c(n)\rho. }
So we choose $\rho=\rho(n)\leq \min\cur{100^{-1},\frac{1}{2} C^{-1}_V(n)c^{-1}(n) }$ and then we have 
\Aligns{\sum_{z\in \mc{C}^{(1,b)}} r_z^k&\leq 1/2,\\
\sum_{z\in \mc{C}^{(1,r_0)}\cup \mc{C}^{(1,f)} } r_z^k &\leq  C_F(n).}
Also for each $y$ in $\mc{C}^{(1,b)}$, we have $r_y<\rho$ and therefore the first step of our inductive covering satisfies the   conditions i-v above.  
\subsubsection*{\it Inductive step} In this step we assume that we have the  properties i-v for step j and we prove then for step j + 1. Basically we only refine the covering for the balls in  $\mc{C}^{(j,b)}$. This is very similar to the first step and we omit the details. See \cite{Naber-Valtorta16}. 
\subsubsection*{\it Conclusion} By property iv above $r_0<r_x\leq \rho^j$ for $x\in \mc{C}^{(j,b)}$. But there will be a step $\bar{j}$ such that $\rho^{\bar{j}}\leq r_0$ and therefore $\mc{C}^{(\bar{j},b)}=\emptyset$.  
\end{proof}

\bibliographystyle{alpha}
\end{document}